\documentclass[12pt]{amsart}

\usepackage{amssymb, amsmath, xcolor}
\usepackage{import}
\usepackage{tikz}
\usepackage{array}
\tikzset{filled/.style={minimum width=5pt,inner sep=0pt,circle,fill=black}}
\usepackage{subcaption}
\usepackage{mathtools}
\usepackage[utf8]{inputenc}
\usepackage{float}
\usepackage{longtable}
\usepackage{todonotes}
\usepackage[square,numbers]{natbib}
\newtheorem{theorem}{Theorem}[section]
\newtheorem{lemma}[theorem]{Lemma}

\newtheorem{question}[theorem]{Question}

\theoremstyle{definition}
\newtheorem{definition}[theorem]{Definition}

\theoremstyle{remark}
\newtheorem{remark}[theorem]{Remark}

\numberwithin{equation}{section}
\numberwithin{figure}{section}

\renewcommand{\mod}{\operatorname{mod}}
\newcommand{\N}{\mathbb{N}}
\newcommand{\Z}{\mathbb{Z}}

\title{Neighborhood Balanced 3-Coloring}
\author[Minyard, Sepanski]{Mitchell Minyard and Mark R. Sepanski}
\date{July 2024}
\begin{document}


\begin{abstract}
    A graph is said to be neighborhood 3-balanced if there exists a vertex labeling with three colors so that each vertex has an equal number of neighbors of each color.
    We give order constraints on 3-balanced graphs, determine which generalized Petersen and Pappus graphs are 3-balanced, discuss when being 3-balanced is preserved under various graph constructions, give two general characterizations of cubic 3-balanced graphs, and classify cubic 3-balanced graphs of small order.
\end{abstract}

\maketitle
\tableofcontents
\section{Introduction}

Freyberg and Marr, \cite{FreybergMarr}, recently introduced the notion of a neighborhood balanced coloring, 
which can be thought of as a two coloring so that each vertex has an equal number of neighbors of each color.
We call graphs that admit such a coloring 2-balanced. In this paper, we study the analogue for three colors and call graphs that admit such a coloring 3-balanced, Definition \ref{def: 3-balanced}.  

The notion of 3-balanced is closely linked to a number of related concepts. 
In \cite{Coupon}, Chen \emph{et.~ al.~}~ introduced a \emph{k-coupon coloring} of a graph, which is 
a $k$-color vertex labeling so that the open neighborhood of each vertex contains all colors.
As a result, every 3-balanced graph gives a 3-coupon coloring, though the converse is false. 

As in \cite{InjectiveChromaticNumber} and \cite{InjectiveGP}, an \emph{injective k-coloring} of a graph is k-color vertex labeling so that no vertices connected by a path of length 2 share the same color. Denote by $\chi_i(G)$ the smallest $k$ such that a injective $k$-coloring exists. As a result, a cubic graph is 3-balanced if and only if $\chi_i(G)= 3$.


The graph coloring problem can also be thought of in the framework of dominating sets. 
A \emph{dominating set} of a graph is a subset of vertices whose neighbors include all other vertices. 
A \emph{total $k$-dominating set} is a dominating set such that every vertex is adjacent to at least $k$ vertices from the set. 
If $G$ is $3k$-regular, then $G$ is 3-balanced if and only $V(G)$ can be partitioned into 3 total $k$-dominating sets. 
Closely related, and studied in \cite{DomaticNumber} and \cite{TotalDomaticNumber}, are the \emph{domatic number} and \emph{total domatic number} of a graph, respectively. The total dominating number, $d_t(G)$, is the maximum
number of total dominating sets into which the vertex set of $G$ can be
partitioned. In particular, a cubic graph is 3-balanced if and only if $d_t(G)=3$.

In \S\ref{sec: notation}, we begin with basic notation and definitions.
Then in \S\ref{sec: order constraints}, we give edge and vertex order constraints for 3-balanced graphs in Theorem \ref{thm: edge number} and additional constraints for regular 3-balanced graphs in Theorem \ref{thm: regular vertex number}.
In \S\ref{sec: Gen Pet} we determine which generalized Petersen graphs are 3-balanced, Theorem \ref{thm: 3-bal col of gen Petersen}, and \S\ref{sec: gen Pappus} does the same for generalized Pappus graphs, Definition \ref{def: gen pappus graph} and Theorem \ref{thm: gen papp graph 3-bal}.
In \S\ref{sec: constructions} we give methods of constructing 3-balanced graphs and study the preservation of the 3-balanced property under certain graph products.
In \S\ref{sec: cubic} we show that cubic 3-balanced graphs admit a Tait coloring in Theorem \ref{def: induced edge coloring}. Then Theorems \ref{thm: cubic sum characterization} and \ref{thm: cubic bijection construction} give characterizations of cubic 3-balanced graphs in terms of summations of edge labels over cycles and bijections between sets, respectively.
In \S\ref{sec: large and small cubic 3-bal} we give forbidden subgraphs and classify small cubic 3-balanced graphs, Theorems \ref{thm: order 6 class} and \ref{thm: cubic on 12} and Lemma \ref{lem: Forbidden Subgraphs}, as well as noting that
the relative frequency of large 3-balanced graphs tends to zero.
We end in \S\ref{sec: concluding remarks} with comments and questions for future work.

\section{Notation and Definitions} \label{sec: notation}

We write $\N$ for the nonnegative integers and $\Z^+$ for the positive integers. When no confusion will arise and when convenient, we allow conflating notation between $\Z$ and $\Z_n$.

We write $G=(V,E)$ for a finite simple undirected graph.
For $v\in V$, we let $N(v)$ denote the \emph{open neighborhood} of $v$ and write $d(v)$ for the \emph{degree} of $v$.

\begin{definition} \label{def: V_i and E_ij}
    Let $\ell:V\rightarrow \Z_3$ be a coloring of a graph, $G$. We write
    \[ V_i = \ell^{-1}(i) \]
    for $i\in\Z_3$ and
    \[ E_{ij} = \{vw\in E \,|\,  \{\ell(v),\ell(w)\} = \{i,j\} \} \]
    for $i,j \in \Z_3$. Note that $E_{ij} = E_{ji}$.
\end{definition}

Motivated by \cite{FreybergMarr}, we have the following definition.
\begin{definition} \label{def: 3-balanced}
    A coloring $\ell:V\rightarrow \Z_3$ of the vertices of a graph, $G$, is said to be \emph{neighborhood 3-balanced}, or \emph{3-balanced} for short, if, for all $v\in V$, 
    \[ |N(v) \cap V_i| = |N(v) \cap V_j| \]
    for all $i,j\in \Z_3$, so that each open neighborhood has an equal number of vertices colored with each color.
\end{definition}

An example of a 3-balanced graph is given in Figure \ref{fig: 3-balanced graph}.
Note that an isolated vertex is trivially 3-balanced. Also, as a graph is 3-balanced if and only if all of its connected components are, so we often restrict our study to connected graphs.

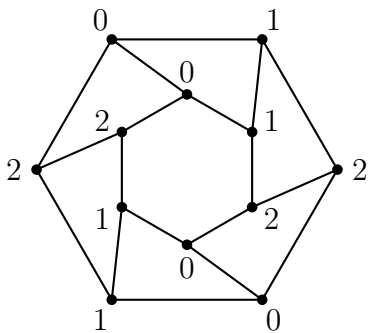
\begin{figure}[H]
    \centering
    
    \begin{tikzpicture}
        
        \coordinate (A) at (0,0);
        \foreach \i/\label in {1/1, 2/0, 3/2, 4/1, 5/0, 6/2} {
          \coordinate (outer\i) at ({60*\i}:2);
          \node at ({60*\i}:2.3) {\label};
        }
        \draw[thick] (outer1) -- (outer2) -- (outer3) -- (outer4) -- (outer5) -- (outer6) -- cycle;
        
        \foreach \i/\label in {1/1, 2/0, 3/2, 4/1, 5/0, 6/2} {
          \coordinate (inner\i) at ({60*\i - 30}:1);
          \node at ({60*\i - 30}:1.3) {\label};
        }
        \draw[thick] (inner1) -- (inner2) -- (inner3) -- (inner4) -- (inner5) -- (inner6) -- cycle;
        
        \foreach \i in {1,...,6} {
          \fill (outer\i) circle (2pt); 
        }
        
        \foreach \i in {1,...,6} {
          \fill (inner\i) circle (2pt); 
        }
        
        \foreach \i in {1,...,6} {
          \draw[thick] (outer\i) -- (inner\i);
        }
    
    \end{tikzpicture}

    \caption{A 3-Balanced Graph}
    \label{fig: 3-balanced graph}
\end{figure}

Note that 3-balanced colorings are not unique. At the very least, the following symmetries exist.
\begin{theorem} \label{thm: easy modification of a 3-bal coloring}
    Let $\ell$ be a 3-balanced coloring of $G$. Then so is $\epsilon \ell + i_0$ for $\epsilon\in\{\pm 1\}$ and $i_0\in\Z_3$.
\end{theorem}

\section{Order Constraints} \label{sec: order constraints}


From Definition \ref{def: 3-balanced} and the Handshaking lemma, we immediately have the following.

\begin{theorem} \label{thm: 3|deg}
    If a graph, $G$, is 3-balanced, then the degree of each vertex is divisible by 3. 
    In particular, $3 \mid |E|$.
\end{theorem}

Note that 3-balanced may certainly have $3\nmid |V|$, see Figure \ref{fig: Non-Regular 3-balanced graph}.

In addition to a constraint on the degree of each vertex, we have the following restriction on the number of edges of each color.

\begin{theorem} \label{thm: edge number}
    Suppose that $G$ is a 3-balanced graph. Then 
    \[|E_{ij}|=\frac{2|E|}{9} \text{ and } |E_{ii}|=\frac{|E|}{9}\]
    for $i,j \in \Z_3$ with $i\not\equiv j$. In particular, $9 \mid |E|$.
\end{theorem}

\begin{proof}
    Assume $G$ has a neighborhood balanced 3-coloring. Then
    \[ |E_{ij}| = \sum_{v\in V_i} \frac{1}{3} d(v)
                =\sum_{v\in V_j} \frac{1}{3} d(v) \]
    for $i,j\in \Z_3$ with $i\not\equiv j$.
    It follows that $|E_{12}|=|E_{13}|=|E_{23}|$. Write $x$ for this common number. 

    By the Handshaking lemma, we see that 
    \[ 2|E| = \sum_{i\in\Z_3} \sum_{v\in V_i} d(v)
            = \sum_{i\in\Z_3} 3x
            = 9x \]
    so that $|E_{ij}| = x = \frac{2|E|}{9}$.
    On the other hand,
    \[ 3x = \sum_{v\in V_i} d(v)
            = 2|E_{ii}| + |E_{i,i+1}| + |E_{i,i+2}|
            = 2|E_{ii}| +2x \]
    so that $|E_{ii}| = \frac{x}{2} = \frac{|E|}{9}$.    
\end{proof}


If $G$ is $r$-regular, we will have additional restrictions. Note that for $r$-regular graphs,
\begin{equation} \label{eqn: r-regular relation}
    |E| = \frac{r |V|}{2}.
\end{equation}
As a result, if $G$ is 3-balanced and $r$-regular, then Theorem \ref{thm: edge number} shows that
\[ 9 \mid (r |V|). \]
More importantly, we have the following relations.

 \begin{theorem} \label{thm: regular vertex number}
    Suppose $G$ is an $r$-regular graph that admits a neighborhood balanced 3-coloring. 
    Then $3\mid r$ 
    and, for $i,j\in\Z_3$ with $i\not\equiv j$,
    \[|E_{ij}|=\frac{r|V|}{9} \text{ and } |E_{ii}|=\frac{r|V|}{18}. \]
    Moreover,
    \[ |V_i| = \frac{|V|}{3}.\]
    Therefore, $3\mid |V|$ and $2 \mid (r|V|)$.
    
    In particular, if $G$ is cubic, \emph{i.e.,} $r = 3$, 
    then 
    \[|E_{ij}|=\frac{|V|}{3} \text{, \,\,} |E_{ii}|=\frac{|V|}{6} 
        \text{, \,\,and \,\,} |V_i| = \frac{|V|}{3} \]
    so that $6\mid |V|$.
\end{theorem}

\begin{proof}
    The claim that $3\mid r$ follows from Theorem \ref{thm: 3|deg}. The first pair of identities follow from Theorem \ref{thm: edge number} and Equation \ref{eqn: r-regular relation}. The next displayed identity follows from the observation that
    \[ \frac{r}{3} \, |V_i| = |E_{i,i+1}| = \frac{r|V|}{9}. \]
    Divisibility by 3 follows from this and divisibility by 2 from Equation \ref{eqn: r-regular relation}.
\end{proof}


We note that the equidistribution of vertex colors in Theorem \ref{thm: regular vertex number} need not hold for non-regular 3-balanced graphs. See Figure \ref{fig: Non-Regular 3-balanced graph} for an example.
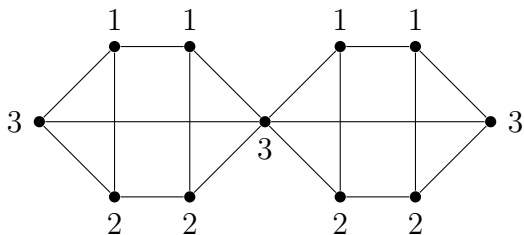
\begin{figure}[H]
    \centering
    
    \begin{tikzpicture}        
        \node[circle, fill, inner sep=1.5pt, label=below:3] (v1) at (0,0) {};
        \node[circle, fill, inner sep=1.5pt, label=above:1] (v2) at (-1,1) {};
        \node[circle, fill, inner sep=1.5pt, label=above:1] (v3) at (-2,1) {};
        \node[circle, fill, inner sep=1.5pt, label=left:3] (v4) at (-3,0) {};
        \node[circle, fill, inner sep=1.5pt, label=below:2] (v5) at (-2,-1) {};
        \node[circle, fill, inner sep=1.5pt, label=below:2] (v6) at (-1,-1) {};
        \node[circle, fill, inner sep=1.5pt, label=above:1] (u2) at (1,1) {};
        \node[circle, fill, inner sep=1.5pt, label=above:1] (u3) at (2,1) {};
        \node[circle, fill, inner sep=1.5pt, label=right:3] (u4) at (3,0) {};
        \node[circle, fill, inner sep=1.5pt, label=below:2] (u5) at (2,-1) {};
        \node[circle, fill, inner sep=1.5pt, label=below:2] (u6) at (1,-1) {};

        \draw (v1) -- (v2);        \draw (v2) -- (v3);        \draw (v3) -- (v4);        \draw (v4) -- (v5);        \draw (v5) -- (v6);        \draw (v6) -- (v1);        \draw (v1) -- (v4);        \draw (v2) -- (v6);        \draw (v3) -- (v5);        \draw (v1) -- (u2);        \draw (u2) -- (u3);        \draw (u3) -- (u4);        \draw (u4) -- (u5);        \draw (u5) -- (u6);        \draw (u6) -- (v1);        \draw (v1) -- (u4);        \draw (u2) -- (u6);        \draw (u5) -- (u3);
        \end{tikzpicture}

    \caption{Non-Regular 3-Balanced Graph With Non-Equidistribution of Colors }
    \label{fig: Non-Regular 3-balanced graph}
\end{figure}

\section{3-Balanced Generalized Petersen Graphs} \label{sec: Gen Pet}

Write $G(m,j)$ for the \emph{generalized Petersen graph} where $m,j\in\N$ with $m\geq 5$ and $1\leq j < \frac{m}{2}$. We will use the notation $\{v_i, u_i \,|\, i\in \Z_m\}$ for the set of vertices with edges
\[ v_i v_{i+1}, \,\, v_i u_i, \text{ and } u_i u_{i+j}. \]
We may refer to the $\{v_i\}$ as the \emph{exterior vertices} and the $\{u_i\}$ as the \emph{interior vertices.}
Observe that the interior vertices break up into $(j,m)$ cycles of size $\frac{m}{(m,j)}$.

We begin with minimal constraints for $G(m,j)$ to admit a 3-balanced coloring.

\begin{lemma} \label{lem: basic Petersen 3-balanced easy result}
    Let $m,j\in\N$ with $m\geq 5$ and $1\leq j < \frac{m}{2}$. 
    \begin{enumerate}
        \item If $G(m,j)$ admits a balanced 3-coloring, then $3\mid m$.
        \item If $3\mid m$, but $3\nmid j$, then $G(m,j)$ admits a balanced 3-coloring.
    \end{enumerate}
\end{lemma}

\begin{proof}
    As $G(m,j)$ is 3-regular and has $2m$ vertices, the first result follows from Theorem \ref{thm: regular vertex number}.

    For the second result, use the coloring 
    $\ell(v_i)=\ell(u_i) =i$. This coloring results in $N(v_i)$ being colored with 
    $\{i-1,i,i+1\}$, which is 3-balanced. It also results in $N(u_i)$ being colored with 
    $\{i-j,i,i+j\}$. These colors overlap if and only if $j \equiv 0$. As $3\nmid j$, this does not happen and the neighborhood is 3-balanced. 
\end{proof}

In fact, we will see that Lemma \ref{lem: basic Petersen 3-balanced easy result} gives a complete description of the generalized Petersen graphs that admit a neighborhood balanced 3-coloring. In \cite{InjectiveGP}, Li, Shao, and Zhu studied \emph{injective k-colorings} and showed that $\chi_i(G(m,j))=3$ if and only if $3 \mid m$ and $3 \nmid j$. As noted in the Introduction, a cubic graph $G$ is 3-balanced if and only if $\chi_i(G)=3$. Below, we present a more direct proof for the classification of 3-balanced generalized Petersen graphs that generalizes nicely to other graph families, see \S\ref{sec: gen Pappus} below for a similar result on generalized Pappus graphs.

\begin{theorem} \label{thm: 3-bal col of gen Petersen}
    Let $m,j\in\N$ with $m\geq 5$ and $1\leq j < \frac{m}{2}$.
    Then $G(m,j)$ admits a balanced 3-coloring if and only if $3\mid m$ and $3\nmid j$.
\end{theorem}

The proof of Theorem \ref{thm: 3-bal col of gen Petersen} will given in Lemmas \ref{lem: matrix inverse}, \ref{lem: soln to M eqn}, and \ref{lem: final Petersen result} below. By Lemma \ref{lem: basic Petersen 3-balanced easy result}, it remains to show that when $3\mid m$ and $3\mid j$, then $G(m,j)$ does not admit a 3-balanced coloring. For example, $G(24,6)$ in Figure \ref{fig: G(24,6) not 3-balanced}, will not admit a balanced 3-coloring. 

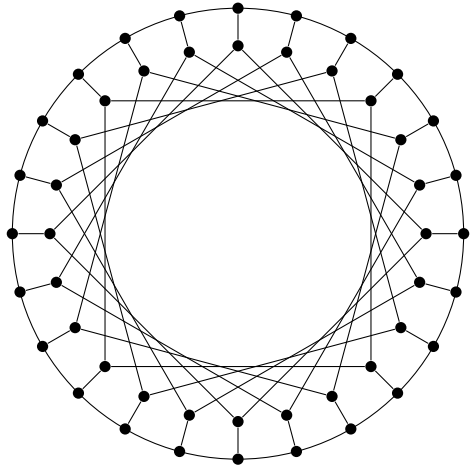
\begin{figure}[H]
    \centering    
    \begin{tikzpicture}
          \def\n{24} 
          \def\m{24} 
          \def\radius{3cm} 
          \def\innerRadius{2.5cm} 
          
          \draw (0,0) circle (\radius);
          \foreach \i in {1,...,\n} {
            \node[circle, fill, inner sep=1.5pt] (outer\i) at ({360/\n * (\i - 1)}:\radius) {};
          }
          
          \foreach \j in {1,...,\m} {
            \node[circle, fill, inner sep=1.5pt] (inner\j) at ({360/\m * (\j - 1)}:\innerRadius) {};
          }
          
          \foreach \k in {1,...,\m} {
            \draw (inner\k) -- (outer\k);
          }
          
          \foreach \k in {1,...,\m} {
            \pgfmathtruncatemacro{\nextdot}{mod(\k + 6 - 1, \m) + 1}
            \draw (inner\k) -- (inner\nextdot);
          }
    \end{tikzpicture}
    
    \caption{$G(24,6)$ Is Not 3-Balanced}
    \label{fig: G(24,6) not 3-balanced}
\end{figure}

Throughout the remainder of this section, suppose $m,j\in\N$ with $m\geq 5$ , $1\leq j < \frac{m}{2}$, and
\[ 3\mid m \text{ and } 3\mid j. \]
Fix 
\[ n = 3^a, \]
with $a\in\Z^+$, so that 
\[  n \mid m.\]

Consider the circulant matrices $A,B\in M_{n\times n}(\Z)$ given by
\[  A = \begin{pmatrix} 
        0 & 1 & 0 & 0 &0 &\ldots & 0& 0 & 1 \\
        1 & 0 & 1 & 0 &0 &\ldots & 0& 0 & 0\\
        0 & 1 & 0 & 1 &0 &\ldots & 0& 0 & 0\\
        &&&&&\vdots&&& \\
        1 & 0 & 0  &0 &0 &\ldots  & 0& 1 & 0
    \end{pmatrix}\]
and 
\[  B = \begin{pmatrix} 
        0 &  \ldots & 0 & 1& 0 &0&0&\ldots & 0 & 1 & 0 & 0 & 0 \ldots & 0\\
        0 &  \ldots & 0 & 0& 1 &0&0&\ldots & 0 & 0 & 1 & 0 & 0 \ldots & 0\\
        0 &  \ldots & 0 & 0& 0 &1&0&\ldots & 0 & 0 & 0 & 1 & 0 \ldots & 0\\
        &&&&&&&\vdots&&&&&&
    \end{pmatrix}\]
with 1's in the first row of $B$ occur in columns $(j+1)$ and $(n-j+1)$. When appropriate, view the column indexing mod $n$ so that, 
for $n=3$ and $j=3$, $B$ degenerates to
\[ \begin{pmatrix}
    2 & 0 & 0\\
    0 & 2 & 0\\
    0 & 0 & 2
\end{pmatrix}.\]
Finally, let $M\in M_{2n\times 2n}(\Z)$ be given by
\[  M = \begin{pmatrix}
        A & I_n \\
        I_n & B
\end{pmatrix}\]
where $I_n$ is the identity matrix of size $n$.

We will need the following result from Mann, \cite{Mann}.
\begin{theorem}[\cite{Mann}, Theorem 1] \label{thm: Mann}
    Let $r,a_1,\ldots,a_r\in\Z^+$ and set 
    \[ m= \prod_{p \text{ prime, } p\leq r} p.\]
    If $\zeta_1, \dots, \zeta_r$ are roots of unity so that 
    \[\sum_{i=1}^r a_i\zeta_i =0, \] 
    but no proper subset, $S$, of $\{1, \dots, r\}$ satisfies 
    $\sum_{i\in S} a_i\zeta_i =0$, 
    then 
    \[ \left( \frac{\zeta_i}{\zeta_j} \right)^m =1 \] 
    for all $1\leq i,j \leq r$.
\end{theorem}

We begin the proof of Theorem \ref{thm: 3-bal col of gen Petersen}. 

\begin{lemma} \label{lem: matrix inverse}
    If $3\mid j$, then
    the matrix, $M$, is nonsingular.
\end{lemma}

\begin{proof}
    Observe that
    \[ M
        \begin{pmatrix}
            -B & I_n \\
            I_n & -A
        \end{pmatrix}
        =
        \begin{pmatrix}
            I_n-AB & 0 \\
            0 & I_n -BA
        \end{pmatrix}.\]
    As a result, it suffices to show that $I_n-AB=I_n-BA$ is nonsingular.

    Write $\omega = e^{2\pi i/n}$. As $A$ and $B$ are circulant, 
    their eigenvectors are
    \[v_k = (1, \omega^k, \omega^{2k}, \omega^{3k}, \dots, \omega^{(n-1)k})\]
    for $k=0,1,\dots,(n-1)$. The corresponding eigenvalues are
    \[\lambda_k^A = \omega^k + \omega^{(n-1)k}=\omega^k + \omega^{-k}\]
    \[\lambda_k^B=\omega^{jk}+\omega^{(n-j)k}=\omega^{jk}+\omega^{-jk},\]
    respectively.
    
    Therefore the eigenvalues of $I_n-AB$ are
     \[ 1-\lambda_k^A\lambda_k^B 
        = 1 - (\omega^{(j+1)k}+\omega^{-(j+1)k}+\omega^{(j-1)k}+\omega^{-(j-1)k}).\]
    So, $I_n-AB$ is singular only when there is a $k$ so that
    \begin{equation} \label{eqn: eq2} 
        0 = -1 + \omega^{(j+1)k}+\omega^{-(j+1)k}+\omega^{(j-1)k}+\omega^{-(j-1)k}
    \end{equation}
    
     Now if Equation \ref{eqn: eq2} holds and no proper subset of the right hand side sums to $0$, then Theorem \ref{thm: Mann} 
     shows that $\omega^{k(j+1)}$ and $\omega^{k(j-1)}$ are 30th roots of unity. 
     However, it is straightforward to verify, by hand or by computer, that 
     the only set of 30th roots of unity, $\{\zeta_1, \zeta_2\}$, that satisfy
     \begin{equation} \label{eqn: roots of zero sum for Petersen}
        1=\zeta_1 +\zeta_1^{-1}+\zeta_2+\zeta_2^{-1}
    \end{equation}
    is $\{1, e^{2\pi i/3}\}$ or its conjugate. As a result, either $\zeta_1$ or $\zeta_2$ is 1. 
    This means that $\omega^{(j+\epsilon)k}=1$ for some $\epsilon \in \{-1,1\}$.
    In turn, this necessitates that $\frac{k(j+\epsilon)}{n}\in \Z$. 
    Therefore, $k(j+\epsilon)\in n\Z = 3^a\Z$. As Equation \ref{eqn: eq2} requires $k\not=0$, we see that $k\in \{1,\dots, n-1\}$, so that $3^a\nmid k$. This requires $3\mid (j+\epsilon)$ which contradicts the fact that $3\mid j$. 

    It remains to check if a proper subsets of Equation \ref{eqn: eq2} can sum to zero. If a proper subset of Equation \ref{eqn: eq2} sums to zero, then by reality and the fact that $1$ is already ruled out, Equation \ref{eqn: roots of zero sum for Petersen} would be changed to 
    \begin{equation*}
        1=\zeta_1 +\zeta_1^{-1} \text{  and  } 0 = \zeta_2+\zeta_2^{-1}.
    \end{equation*}
    In this case, we would have 
    $\zeta_1 = e^{\pm 2\pi i/ 6}$ and $\zeta_2 = i^{\pm 1}$.
    This can be reduced to having
    $\frac{k(j+\epsilon)}{n} \in \frac{1}{6} + \Z$ and 
    $\frac{k(j-\epsilon)}{n} \in \frac{1}{4} + \Z$.  
    The latter implies that 
    $3^a \mid (4k(j-\epsilon))$ so that $3^a \mid k$, which is also a contradiction. 
\end{proof}

From Lemma \ref{lem: matrix inverse}, we immediately get the following result.

\begin{lemma} \label{lem: soln to M eqn}
    The unique solution to the equation 
    \[Mx = \frac{m}{n} \begin{pmatrix}
        1\\ \vdots \\ 1
    \end{pmatrix}\]
    is 
    \[x=\frac{m}{3n}\begin{pmatrix}
        1\\ \vdots \\ 1
    \end{pmatrix}\]    
\end{lemma}

Next is the final step of the proof of Theorem \ref{thm: 3-bal col of gen Petersen}. 

\begin{lemma} \label{lem: final Petersen result}
    When $3\mid m$ and $3\mid j$, then $G(m,j)$ does not admit a 3-balanced coloring.
\end{lemma}

\begin{proof}
    By way of contradiction, assume $G(m,j)$ is equipped with a 3-balanced coloring, $\ell$. We will obtain a contradiction by showing that, whenever $3^a\mid m$, then $3^{a+1}\mid m$. Recall $n=3^a$.

    Fix $\alpha_0 \in \Z_3$. For each $i\in\Z_n$,
    we will count the number of exterior vertices $v_j$ and interior vertices, $u_j$, with $j\equiv i \mod n$ such that $v_j$, respectively $u_j$, is labeled by $\alpha_0$.
    Precisely, for $i\in \{0,1,2,\ldots, n-1\}$, let
    \[x_i= |\ell^{-1}(\alpha_0) \cap
            \{v_j \mid j\in i+n\Z_m\}| \] 
    \[y_i= |\ell^{-1}(\alpha_0) \cap
            \{u_j \mid j\in i+n\Z_m\}|. \] 
    From 3-balanced, we see that
    \[\frac{m}{n}
        =\sum_{j\in i+n\Z_m} |\ell^{-1}(\alpha_0) \cap N(v_j)|
        = x_{i-1} + x_{i+1} + y_{i}\]
    \[\frac{m}{n}
        =\sum_{j\in i+n\Z_m} |\ell^{-1}(\alpha_0) \cap N(u_j)|
        = y_{i-j} + y_{i+j} + x_{i}.\]
    Let $\nu = \begin{pmatrix}
        x_1,\, \dots,\,  x_n,\, y_1,\, \dots,\, y_n
    \end{pmatrix}^T$. It follows that 
    \[M\nu = \frac{m}{n}\begin{pmatrix}
        1\\ \vdots \\ 1
    \end{pmatrix}.\]
    By Lemma \ref{lem: soln to M eqn}, it follows that 
    \[\nu = \frac{m}{3n} \begin{pmatrix}
        1\\ \vdots \\ 1
    \end{pmatrix}.\]
    As $x_i, y_i \in\Z$, we see that $3n\mid m$ so that $3^{a+1}\mid m$ and we are done.
\end{proof}

\section{3-Balanced Generalized Pappus Graphs} \label{sec: gen Pappus}

The \emph{Pappus graph} is a symmetric cubic graph on 18 vertices that is easily seen to be 3-balanced in Figure \ref{fig: Pappus Graph}. A \emph{generalized Pappus graph} was introduced in \cite{GeneralizedPappus}. We further extend their definition below.
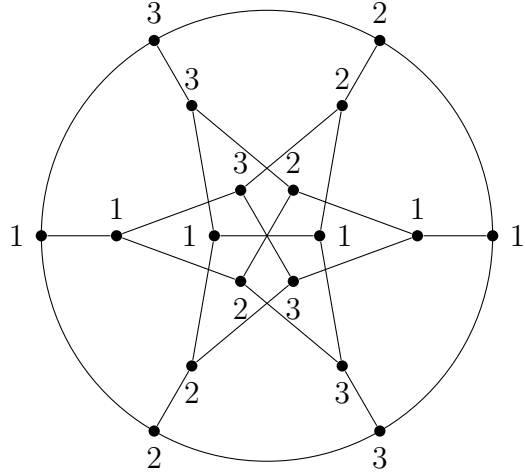
\begin{figure}[H]
\centering
        \begin{tikzpicture} \label{P(18,3,9)}
    \def\n{6} 
    \def\m{6} 
    \def\l{6} 
    \def\radius{3cm} 
    \def\middleRadius{2cm} 
    \def\innerRadius{.7cm} 
          
    \draw (0,0) circle (\radius);
    \node[circle, fill, inner sep=1.5pt, label=right:1] (v1) at ({360/\n * (0)}:\radius) {};
    \node[circle, fill, inner sep=1.5pt, label=above:2] (v2) at ({360/\n * (1)}:\radius) {};
    \node[circle, fill, inner sep=1.5pt, label=above:3] (v3) at ({360/\n * (2)}:\radius) {};
    \node[circle, fill, inner sep=1.5pt, label=left:1] (v4) at ({360/\n * (3)}:\radius) {};
    \node[circle, fill, inner sep=1.5pt, label=below:2] (v5) at ({360/\n * (4)}:\radius) {};
    \node[circle, fill, inner sep=1.5pt, label=below:3] (v6) at ({360/\n * (5)}:\radius) {};
              
    \node[circle, fill, inner sep=1.5pt,label=above:1] (u1) at ({360/\m * (0)}:\middleRadius) {};
    \node[circle, fill, inner sep=1.5pt, label=above:2] (u2) at ({360/\m * (1)}:\middleRadius) {};
    \node[circle, fill, inner sep=1.5pt, label=above:3] (u3) at ({360/\m * (2)}:\middleRadius) {};
    \node[circle, fill, inner sep=1.5pt,label=above:1] (u4) at ({360/\m * (3)}:\middleRadius) {};
    \node[circle, fill, inner sep=1.5pt, label=below:2] (u5) at ({360/\m * (4)}:\middleRadius) {};
    \node[circle, fill, inner sep=1.5pt, label=below:3] (u6) at ({360/\m * (5)}:\middleRadius) {};

    \node[circle, fill, inner sep=1.5pt, label=right:1] (w1) at ({360/\l * (0)}:\innerRadius) {};
    \node[circle, fill, inner sep=1.5pt, label=above:2] (w2) at ({360/\l * (1)}:\innerRadius) {};
    \node[circle, fill, inner sep=1.5pt, label=above:3] (w3) at ({360/\l * (2)}:\innerRadius) {};
    \node[circle, fill, inner sep=1.5pt, label=left:1] (w4) at ({360/\l * (3)}:\innerRadius) {};
    \node[circle, fill, inner sep=1.5pt, label=below:2] (w5) at ({360/\l * (4)}:\innerRadius) {};
    \node[circle, fill, inner sep=1.5pt, label=below:3] (w6) at ({360/\l * (5)}:\innerRadius) {};
    
    \draw (v1) -- (u1);    \draw (v2) -- (u2);    \draw (v3) -- (u3);    \draw (v4) -- (u4);    \draw (v5) -- (u5);    \draw (v6) -- (u6); 
    \draw (u1) -- (w2);    \draw (u1) -- (w6);    \draw (u2) -- (w1);    \draw (u2) -- (w3);    \draw (u3) -- (w2);    \draw (u3) -- (w4);    \draw (u4) -- (w3);    \draw (u4) -- (w5);    \draw (u5) -- (w4);    \draw (u5) -- (w6);    \draw (u6) -- (w1);    \draw (u6) -- (w5);
    \draw (w1) -- (w4);    \draw (w5) -- (w2);    \draw (w3) -- (w6);
    \end{tikzpicture}
    \caption{Pappus Graph, $P(6,1,3)$}
    \label{fig: Pappus Graph}
\end{figure}

\begin{definition} \label{def: gen pappus graph}
    Let $m,j,k \in \N$ with $m\geq4$, $1\leq j <\frac{m}{2}$, and $1\leq k\leq \frac{m}{2}$.
    The \emph{generalized Pappus graph}, written $P(m,j,k)$, has vertex set 
    $\{v_i, u_i, w_i \mid i\in \Z_m\}$ and edges
    \[ v_i v_{i+1}, \,\, v_i u_i, \,\, u_i w_{i+j}, \,\, u_i w_{i-j} \text{ and } w_i w_{i+k}. \]
\end{definition}

We may refer to the $\{v_i\}$ as the \textit{exterior vertices}, the $\{u_i\}$ as the \textit{middle vertices}, and the $\{w_i\}$ as the \textit{interior vertices}. Notice that the degree of each vertex of $P(m,j,k)$ is divisible by $3$ if and only if $P(m,j,k)$ is cubic if and only if \[ k=\frac{m}{2}.\]
Therefore, a generalized Pappus graph that is 3-balanced necessarily has $m$ even and $k=\frac{m}{2}$. In fact, we next show that it is further necessary to require $6\mid m$.

\begin{lemma}
    Let $m,j \in \N$ with $m\geq4$, $m$ even, and $1\leq j <\frac{m}{2}$. If $P(m,j,\frac{m}{2})$ admits a balanced 3-coloring, then $3\mid m$.
\end{lemma}

\begin{proof}
    Assume $P(m,j,\frac{m}{2})$ is equipped with a 3-balanced coloring, $\ell$. 
    Fix $\alpha_0 \in \Z_3$. For each $i\in\Z_2$,
    we will count the number of exterior vertices
    $v_j$ with $j\equiv i \mod 2$ and $v_j$
    labeled by $\alpha_0$.
    We will similarly count the middle vertices $u_j$, and the interior vertices $w_j$.
    Precisely, let
    \[x_i= |\ell^{-1}(\alpha_0) \cap
            \{v_j \mid j\in i+2\Z_m\}| \] 
    \[y_i= |\ell^{-1}(\alpha_0) \cap
            \{u_j \mid j\in i+2\Z_m\}| \] 
    \[z_i= |\ell^{-1}(\alpha_0) \cap
            \{w_j \mid j\in i+2\Z_m\}|. \] 
    From 3-balanced, we see that
    \[\frac{m}{2}
        =\sum_{j\in i+2\Z_m} |\ell^{-1}(\alpha_0) \cap N(v_j)|
        =  2x_{i+1} + y_{i}\]
    \[\frac{m}{2}
        =\sum_{j\in i+2\Z_m} |\ell^{-1}(\alpha_0) \cap N(u_j)|
        = 2z_{i+j} + x_{i}\]
     \[\frac{m}{2}
        =\sum_{j\in i+2\Z_m} |\ell^{-1}(\alpha_0) \cap N(w_j)|
        = 2y_{i+j} + z_{i}\]
    This yields 6 equations each summing to $\frac{m}{2}$. 
    It is straightforward to verity that the only solution to these equations is
    \[v_0=v_1=u_0=u_1=w_0=w_1=\frac{m}{6}.\]
    From this we see that $3 \mid m$.
    \end{proof}

Next we show that when $6\mid m$, but $3\nmid j$, then $P(m,j,\frac{m}{2})$ is 3-balanced.

\begin{lemma}
    If $6\mid m$, but $3 \nmid j$, then $P(m,j,\frac{m}{2})$ admits a balanced 3-coloring.
\end{lemma}

\begin{proof}
    We use the coloring 
    $\ell(v_i)=\ell(u_i) =\ell(w_i)=i$. This coloring results in $N(v_i)$ being colored with 
    $\{i-1,i,i+1\}$, which is 3-balanced. It also results in $N(u_i)$ being colored with 
    $\{i-j,i,i+j\}$. These colors overlap if and only if $j \equiv 0$. As $3\nmid j$, this does not happen and the neighborhood is 3-balanced. Lastly, $N(w_i)$ is colored with $\{i+\frac{m}{2}, i-j,i+j\}$. Since $6\mid m$, we have $i+\frac{m}{2} \equiv i$, so this neighborhood is 3-balanced.
\end{proof}

Throughout the remainder of this section, suppose 
$m,j\in\N$ with $m\geq 4$, 
$1\leq j < \frac{m}{2}$, and
\[ 6\mid m \text{ and } 3\mid j. \]
We will eventually show that $P(m,j,\frac{m}{2})$ is not 3-balanced in this case.

Fix 
\[ n = 3^a, \]
with $a\in\Z^+$, so that 
\[  n \mid m.\]
Now consider the circulant matrices $A,B \in M_{n\times n}(\Z)$ as in \S\ref{sec: Gen Pet} and additionally consider the circulant matrix $C \in M_{n\times n}(\Z)$ given by
\[  C = \begin{pmatrix} 
        0 &  \ldots & 0 & 1& 0 &0&0&\ldots & 0 \\
        0 &  \ldots & 0 & 0& 1 &0&0&\ldots & 0 \\
        0 &  \ldots & 0 & 0& 0 &1&0&\ldots & 0 \\
        &&&&\vdots
    \end{pmatrix}\]
where the 1 in the first row is in column $\frac{m}{2}+1$.
Finally, let $L\in M_{3n\times 3n}(\Z)$ be given by
\[  L = \begin{pmatrix}
        A & I_n & 0_n\\
        I_n & 0_n & B\\
        0_n & B & C
\end{pmatrix}\]
where $0_n$ is the zero matrix of size $n$.

\begin{lemma}
    If $3\mid j$, then
    the matrix, $L$, is nonsingular.
\end{lemma}

\begin{proof}
    By row-reducing, we can immediately see that $L$ is equivalent to 
    \[   \begin{pmatrix}
        I_n & 0 & B\\
        0 & I_n & -AB\\
        0 & 0 & BAB+C
\end{pmatrix}.\]
    Therefore, it suffices to show that $BAB+C$ is nonsingular.

    Write $\omega = e^{2\pi i/n}$. As $A, B$ and $C$ are circulant, 
    their eigenvectors are
    \[v_k = (1, \omega^k, \omega^{2k}, \omega^{3k}, \dots, \omega^{(n-1)k})\]
    for $k=0,1,\dots,(n-1)$. The corresponding eigenvalues are
    \[\lambda_k^A = \omega^k + \omega^{(n-1)k}=\omega^k + \omega^{-k}\]
    \[\lambda_k^B=\omega^{jk}+\omega^{(n-j)k}=\omega^{jk}+\omega^{-jk},\]
    \[\lambda_k^C=\omega^{(n-\frac{m}{2})k}=\omega^{-\frac{mk}{2}},\]
    respectively.
    
    Therefore the eigenvalues of $BAB+C$ are
    \begin{equation} \label{eqn: Pap preliminary evals}
        \lambda_k^B\lambda_k^A\lambda_k^B+\lambda_k^C=(\omega^k + \omega^{-k})(\omega^{kj}+\omega^{-kj})^2+\omega^{-\frac{mk}{2}}
    \end{equation}
    and $BAB+C$ is singular if and only if there exists a $k\in \{0,\dots,n-1\}$ for which
    Equation \ref{eqn: Pap preliminary evals} is zero.    
    Now, as $n$ is odd, there is no $k$ such that $\omega^{-\frac{mk}{2}}=-1$.
    Therefore, as $(\omega^k + \omega^{-k})(\omega^{kj}+\omega^{-kj})^2$ is real, making Equation \ref{eqn: Pap preliminary evals} zero requires
    $\omega^{-\frac{mk}{2}}=1$ and $1 + (\omega^k + \omega^{-k}) \allowbreak (\omega^{kj}+\omega^{-kj})^2  = 0$.
    Expanding we see that $BAB+C$ is singular exactly when there is a $k$ such that
    \begin{equation} \label{eqn: Pappus omega} 
        0 =1+ \omega^{k(2j+1)}+  \omega^{-k(2j+1)}
         + \omega^{k(2j-1)} + \omega^{-k(2j-1)}
        +2\omega^{k}+2\omega^{-k}
    \end{equation}
    
     Now if Equation $\ref{eqn: Pappus omega}$ holds and no proper subset sums to 0, Theorem \ref{thm: Mann} shows that $\omega^{k}$, $\omega^{k(2j+1)}$, and $\omega^{k(2j-1)}$ are 210th roots of unity. 
     However, it is straightforward to verify by computer, that, up to conjugation and relabeling of $i=1$ and $i=2$, there are three sets of 210th roots of unity, $\zeta_1, \zeta_2, \zeta_3$, that satisfy
    \begin{equation}\label{eq: Pappus zeta}
         0=1+\zeta_1 +\zeta_1^{-1}+\zeta_2+\zeta_2^{-1}+2\zeta_3+2\zeta_3^{-1}
     \end{equation}
    and all are of the form $\{\omega^a, \omega^b, \omega^c\}$ with $a,b,c \in \{0, 35, 70, 105\}$.
    The three sets of solutions are 
    $\{105, 70, 35 \}$, $\{70, 0, 70 \}$, and $\{35, 0, 105 \}$.
    It is straightforward to show that each of these possibilities lead to a contradiction under the hypothesis that $3\mid j$. For example, in the case of $\{105, 70, 35 \}$, 
    looking at Equation \ref{eqn: Pappus omega}, shows that $k\equiv 35$ and that
    $35(2j+\epsilon-2)\equiv0$ and $35(2j-\epsilon-3)\equiv 0$ $\mod n$ for some $\epsilon\in\{\pm 1\}$. As $n=3^a$, at the very least, this implies that $3\mid (\epsilon -2)$ and $3\mid (-\epsilon)$ which gives a contradiction. The other cases are handled similarly.     
    
    We now turn to the case that proper subsets of Equations \ref{eqn: Pappus omega} and \ref{eq: Pappus zeta} sum to 0.
    Define $S$ to be the smallest proper subset of Equation \ref{eq: Pappus zeta} that includes 1, and $S^c$ to be its complement. Therefore the elements of $S$ and $S^c$ sum to 0 and $2 \leq |S| \leq 5$.
    
    First, observe that $|S|\neq 2$, as $n$ is odd, and therefore $\zeta_i \neq -1$. 
    In a similar fashion, we record a number of related forbidden configurations for future use. 
    The first is that     
    \begin{equation}\label{eq: pappus case 1}
        \zeta_i+\zeta_i^{-1}=0
    \end{equation} 
    is not possible for $i=1,2,3$ since, again, $n$ is odd so that $\zeta_i \not= \pm i$.  
    Next, note that we can never have 
    \begin{equation}\label{eq: pappus case 3}
        2\zeta_3+2\zeta_3^{-1}=\epsilon, 
    \end{equation}
    for $\epsilon \in \{\pm1\}$, as the solutions to 
    the associated quadratic equation (multiplying by $\zeta_3$) are easily seen not to be 6th roots of unity as required by Theorem \ref{thm: Mann}.
    Finally, we cannot have 
    \begin{equation}\label{eq: pappus case 2}
        \zeta_i+\zeta_i^{-1}=-1
    \end{equation}
    as $\zeta_i$ would be a 3rd root of unity, which can only occur when
    $\frac{k(2j+\epsilon)}{n} \in \pm\frac{1}{3} + \Z$.
    In turn, this is the same as
    $k(2j+\epsilon) \equiv \pm\frac{n}{3} \mod n$.
    As $n=3^a$, this implies that $3^{a-1} \mid k(2j+\epsilon)$. Since $3 \mid j$, $3 \nmid (2j+\epsilon)$ and so $3^{a-1} \mid k$. Since $0 < k < 3^a$, we see that 
    $k \in \{ 3^{a-1}, 2\cdot 3^{a-1} \} = \{ \frac{n}{3}, \frac{2n}{3} \}$.
    This heavily restricts Equation \ref{eq: Pappus zeta}, as for these $k$, 
    $2\zeta_3+2\zeta_3^{-1}=-2$. Therefore, Equation \ref{eq: Pappus zeta} requires 
    $\zeta_{i+1}+\zeta_{i+1}^{-1}=2$ 
    which is only possible when $k=0$ and is a contradiction.
    
    Consider the case $|S|=3$. 
    Equations \ref{eq: pappus case 3} and \ref{eq: pappus case 2} show that the only possibility would require
    \[1 +\zeta_i +2\zeta_3=0.\]
    In this case, all roots of unity involved are 6th roots of unity by Theorem \ref{thm: Mann}.
    It is straightforward to verify by hand or by computer program that all such solutions
    have $\zeta_3 \in \{\pm i\}$, which is impossible as $n$ is odd.
    
    Turn now to $|S|=4$. 
    If $S^C$ contains a conjugate pair, then the third element of $S^C$ is real, and thus the conjugate pair sums to $\pm 1$ or $\pm 2$. 
    The sentence before Equation \ref{eq: pappus case 1} rules out $1$ and Equation \ref{eq: pappus case 2} rules out $-1$.
    The conjugate pair can only sum to $2$ if $k=0$, which violates Equation \ref{eq: Pappus zeta}, and the cannot sum to $-2$ as $n$ is odd. 
    As a result, $S^C$ must contains one of each of $\{\zeta_1^{\pm1}, \zeta_2^{\pm 1}, 2\zeta_3^{\pm1}\}$ and $S$ consists of $1$ and the conjugates of $S^C$. As the sum of the elements of $S$ is $0$, this forces the sum of the elements of $S^C$ to be $-1$, which is a contradiction. 
    
    Finally, suppose $|S|=5$. But then $|S^c|=2$, so by Equation \ref{eq: pappus case 1} we only need to consider the case of
    \[\zeta_i +2\zeta_3=0.\]
    However, this would force $\zeta_i/\zeta_3^{-1} = -2$, which is impossible.
\end{proof}

Now note that Lemma \ref{lem: soln to M eqn} holds for the matrix $L$ in place of $M$. 

\begin{lemma}
    When $6\mid m$ and $3\mid j$, then $P(m,j, \frac{m}{2})$ does not admit a 3-balanced coloring.
\end{lemma}

\begin{proof}
    By way of contradiction, assume $P(m,j, \frac{m}{2})$ is equipped with a 3-balanced coloring, $\ell$. We will obtain a contradiction by showing that, whenever $3^a\mid m$, then $3^{a+1}\mid m$. Recall $n=3^a$.

    Fix $\alpha_0 \in \Z_3$. As before, for $i\in\Z_n$,
    we will count the number of exterior, middle, and interior vertices, 
    $v_j$, $u_j$, and $w_j$ labeled by $\alpha_0$ with $j\equiv i\mod n$
    Precisely, for $i\in \{0,1,2,\ldots, n-1\}$, let
    \[x_i= |\ell^{-1}(\alpha_0) \cap
            \{v_j \mid j\in i+n\Z_m\}| \] 
    \[y_i= |\ell^{-1}(\alpha_0) \cap
            \{u_j \mid j\in i+n\Z_m\}| \] 
    \[z_i= |\ell^{-1}(\alpha_0) \cap
            \{w_j \mid j\in i+n\Z_m\}|. \] 
    From 3-balanced, we see that
    \[\frac{m}{n}
        =\sum_{j\in i+n\Z_m} |\ell^{-1}(\alpha_0) \cap N(v_j)|
        = x_{i-1} + x_{i+1} + y_{i}\]
    \[\frac{m}{n}
        =\sum_{j\in i+n\Z_m} |\ell^{-1}(\alpha_0) \cap N(u_j)|
        = z_{i-j} + z_{i+j} + x_{i}.\]
    \[\frac{m}{n}
        =\sum_{j\in i+n\Z_m} |\ell^{-1}(\alpha_0) \cap N(w_j)|
        = y_{i-j} + y_{i+j} + w_{i+\frac{m}{2}}.\]
     Let $\mu = \begin{pmatrix}
        x_1,\, \dots,\,  x_n,\, y_1,\, \dots,\, y_n, \, w_1, \, \dots, \, w_n
    \end{pmatrix}^T$. Then we have
    \[L\mu = \frac{m}{n}\begin{pmatrix}
        1\\ \vdots \\ 1
    \end{pmatrix}.\]
    By Lemma \ref{lem: soln to M eqn}, it follows that 
    \[\mu = \frac{m}{3n} \begin{pmatrix}
        1\\ \vdots \\ 1
    \end{pmatrix}.\]
    As $x_i, y_i, w_i \in\Z$, we see that $3n\mid m$ so that $3^{a+1}\mid m$ and we are done.
\end{proof}

Combining the previous Lemmas we get the following.

\begin{theorem} \label{thm: gen papp graph 3-bal}
    The generalized Pappus graph, $P(m,j,k)$, is 3-balanced if and only if $6\mid m$, $3\nmid j$, and $k=\frac{m}{2}$.
\end{theorem}

\section{Constructions of 3-balanced Graphs} \label{sec: constructions}

In this section, we give methods of constructing 3-balanced graphs and study the preservation of the 3-balanced property under certain graph products. 

We begin with the construction of an infinite family of cubic 3-balanced graphs.
Let $n\in 2\Z$ with $n\geq 4$.
The \emph{M{\"o}bius ladder}, $M_n$, is constructed from an $n$-cycle by adding edges connecting opposite pairs of vertices.
By Theorem \ref{thm: regular vertex number}, a minimum requirement for $M_n$ to be 3-balanced is $6\mid n$.
When $6\mid n$, write $v_1, \dots, v_{6n}$ for the vertices and define a vertex coloring by $\ell(v_i)=i \mod 3$. It is straightforward to see that this labeling is $3$-balanced and we get the following.

\begin{theorem} \label{thm: mobius}
    Let $n\in \Z^+$ with $n\geq 4$.
    The M{\"o}bius Ladder, $M_{n}$, is cubic 3-balanced if and only if $6 \mid n$.
\end{theorem}

We now turn to the trivial observation that the 3-balanced property is preserved under edge disjoint unions if the graphs share a uniform 3-balanced coloring.

\begin{lemma} \label{lem: sum of 3-balanced}
    Let $k\in\Z^+$ and $G_i=(V, E_i)$, $1\leq i \leq k$, with $E_i \cap E_j = \emptyset$, and $\ell:V\to \Z_3$ a coloring. 
    If $\ell$ is 3-balanced for each $G_i$, then $G=(V,\, \bigcup_{i=1}^k E_i)$ is also 3-balanced.
\end{lemma}

Next, we observe that the property of being 3-balanced is preserved under gluing at a single vertex.

\begin{theorem}
    Let $k\in\Z^+$ and $G_i=(V_i, E_i)$, $1\leq i \leq k$, be 3-balanced
    with $V_i \cap V_j = \{v_0\}$ for $1\leq i < j \leq k$.
    Then $G=(\bigcup_{i=1}^k V_i, \, \bigcup_{i=1}^k E_i)$ is 3-balanced.
\end{theorem}

\begin{proof}
    Let $\ell_i$ be a 3-balanced labeling of $V_i$ for $G_i$. By Theorem \ref{thm: easy modification of a 3-bal coloring}, we may assume that $\ell_i(v_0) =0$. It is then straightforward that coloring of $\bigcup_{i=1}^k V_i$ with $\ell_i$ on $G$ is well defined and 3-balanced.
\end{proof}

The following definition is a generalization of join.
\begin{definition} \label{def: join of graphs along a graph}
    Let $G = (V,E)$, and let $\mathcal{G}=\{G_v \mid v\in V\}$ be a collection of graphs indexed by $V$. We define the \emph{join of $\mathcal{G}$ along $G$}, $\nabla\mathcal{G}$, to be the graph constructed as follows. Begin with $\bigcup_{v\in V} G_v$. For every $vw\in E$, include all possible edges between $G_v$ and $G_w$.
    Note that $G_v \nabla G_w$, the join of $G_v$ and $G_w$ for disjoint $v,w\in V$, is a subgraph of $\nabla\mathcal{G}$.
\end{definition}

Next we show that the join along a graph of copies of the complement of complete graphs, $\overline{K_{3n}}$, produces 3-balanced graphs. 

\begin{theorem}
    Let $G=(V,E)$, $n_v \in \Z^+$ for $v\in V$, 
    $G_v = \overline{K_{3n_v}}$, and $\mathcal{G}=\{G_v \mid v\in V\}$. Then $\nabla\mathcal{G}$ is 3-balanced.
\end{theorem}

\begin{proof}
    Choose any coloring of $G_v$, $\ell_v$, that
    colors one third of the vertices of $G_v$ with each color. 
    It is straightforward to see this produces a 3-balanced coloring of 
    $\nabla\mathcal{G}$
\end{proof}

In general, the join of two 3-balanced graphs is not 3-balanced.
For example, let $G$ be the 3-balanced graph from Figure \ref{fig: Non-Regular 3-balanced graph} with $11$ vertices and vertex degrees of 3 and 6. Then $G\,\nabla G$ is not 3-balanced by Theorem \ref{thm: 3|deg} as the degree of each vertex is 14 or 17.
However, if we restrict our focus to regular graphs, we will see below that the join preserves the 3-balanced property.

\begin{theorem}
    Let $i\in\{1,2\}$, $k_i\in\Z^+$, and   
    $G_i$ be $3k_i$-regular graphs.
    Suppose $G_i$ are 3-balanced. Then $G_1 \nabla G_2$ is 3-balanced.
\end{theorem}

\begin{proof}
    Suppose $G_i$ is 3-balanced. 
    Let $\ell_i$ be a 3-balanced coloring of $G_i$.
    Due to Theorem \ref{thm: regular vertex number}, 
    the vertex colors of $G_i$ are equidistributed. It follows that $\ell_1$ and $\ell_2$ induce a 3-balanced coloring $G_1\nabla G_2$.
\end{proof}

Now we move on to graph products and show that the product of 3-balanced graphs is 3-balanced, with some stronger results for the lexicographic product. We write $\Box$ for the \emph{Cartesian product}, $\times$ for the \emph{tensor product}, $\boxtimes$ for the \emph{strong product}, and $\cdot$ for the \emph{lexicographic product}.

\begin{theorem}
    If $G_1$, $G_2$ are 3-balanced, then so are $G_1\Box G_2$, $G_1\times G_2$, and $G_1\boxtimes G_2$.
\end{theorem}

\begin{proof}
    Let $\ell_i$ be 3-balanced labeling for $G_i$, $i=1,2$. 
    Define a labeling, $\ell$, on $V_{G_1}\times V_{G_2}$ by $\ell((u,v)) = \ell_1(u)+\ell_2(v)$.
    Begin with $G = G_1\Box G_2$. Then 
    \[N_G((u, v))=\{(u, w) \mid w \in N_{G_2}(v) \} \cup \{ (w,v) \mid w \in N_{G_1}(u) \}.\] 
    By construction, the colors are equidistributed on each set.
    For $H = G_1 \times G_2$, 
    \[N_H((u,v))=\{ (w,x) \mid w\in N_{G_1}(u), x \in N_{G_2}(v) \}.\] 
    Writing this as $\cup_{x \in N_{G_2}(v)} \{ (w,x) \mid w\in N_{G_1}(u) \}$, each of these subsets have equidistributed colors by construction.
    Lastly, $G_1 \boxtimes G_2$ is 3-balanced by Lemma \ref{lem: sum of 3-balanced} as $E(G_1\times G_2)\cap E(G_1\Box G_2)=\emptyset$.
\end{proof}

\begin{theorem}
    Let $r_1, r_2\in\Z^+$, $G_1$ a $3r_1$-regular graph, and $G_2$ a $2r_2$-regular graph.
    Suppose $G_1$ is 3-balanced and $G_2$ is 2-balanced.
    Then $G_1 \times G_2$ is 3-balanced.
\end{theorem}

\begin{proof}
    Let $\ell_1$ be a 3-balanced labeling of $G_1$ and $\ell_2$ a 2-balanced labeling of $G_2$ colored with $\pm 1$. 
    Define $\ell$ on $V_{G_1}\times V_{G_2}$ by $\ell(u,v)= \ell_1(u) \cdot \ell_2(v)$.
    As $N((u,v))=\cup_{x \in N_{G_2}(v)} \{ (w,x) \mid w\in N_{G_1}(u) \}$, each of these subsets have equidistributed colors by construction.     
\end{proof}

Interestingly, the lexicographic product only requires the second graph to be 3-balanced.


\begin{theorem}
    Let $r\in\Z^+$, $G_1$ a graph, and $G_2$ a $3r$-regular graph.
    Suppose $G_2$ is 3-balanced. Then $G_1\cdot G_2$ is 3-balanced.
\end{theorem}

\begin{proof}
    Let $\ell_2$ be a 3-balanced coloring of $G_2$. Define $\ell$ on $V_{G_1}\times V_{G_2}$ by 
    $\ell((u,v))= \ell_2(v)$. Now notice that
    \[N((u,v)) =\{ (u,w) \mid w\in N_{G_2}(v) \} \cup \{ (w,x) \mid w\in N_{G_1}(u), x\in V_{G_2} \}. \]
    The first set has equidistributed colors by construction. Writing the second set as 
    $\cup_{w\in N_{G_1}(u)} \{ (w,x) \mid x\in V_{G_2} \}$, we see that the second set similarly has equidistributed colors by Theorem \ref{thm: regular vertex number}.
\end{proof}

\section{Characterizations of 3-Balanced Cubic Graphs} \label{sec: cubic}

In this section, every graph $G$ is assumed to be cubic. Recall that a \emph{Tait coloring} of a cubic graph is a proper $Z_3$ edge coloring. A \emph{snark} is a cubic graph with edge chromatic number 4, the only other option by Vizing's theorem. Some authors also add additional connectivity and cycle length requirements to the definition of a snark.

\begin{definition} \label{def: induced edge coloring}
    Let $\ell$ be a coloring of the vertices of $G$. Define the \emph{induced edge coloring}, also denoted by $\ell$, as
    \[ \ell(uv) = \ell(u) + \ell(v) \]
    for $uv\in E$.
\end{definition}

\begin{theorem} \label{thm: perfmatch}
    Let $G$ be cubic and 3-balanced. Then the induced edge coloring is a Tait coloring of $G$ that gives rise to 3 edge-disjoint perfect matchings and, for any choice of two colors, a vertex covering by cycles that alternate between those two colors.
\end{theorem}

\begin{proof}
    Let $\ell$ be a 3-balanced coloring and extend this coloring to the edges via the induced edge coloring.
    By construction, the three edges adjacent to any vertex of $G$ are colored with three distinct colors.
    As a result, all edges colored with the same color give rise to a perfect matching. Moreover, as $G$ is finite and cubic, any choice of two colors result in an exhaustion of the vertices of the graph by cycles that alternate between those two colors.
\end{proof}

Note that the converse to Theorem \ref{thm: perfmatch} is not true. For instance, every generalized Petersen graph, except the Petersen graph, has a Tait coloring (\cite{CastagnaPrins}, \cite{Watkins}), but 
Theorem \ref{thm: 3-bal col of gen Petersen} shows that most generalized Petersen graphs are not 3-balanced.




In particular, Theorem \ref{thm: perfmatch} shows that no snark is 3-balanced. From the covering with cycles part of Theorem \ref{thm: perfmatch}, it follows that 3-balanced graphs are bridgeless (see also \cite{Isaacs}, 2.4).
As edge connectivity is the same as vertex connectivity for cubic graphs, connected 3-balanced graphs are 2-connected. In turn, this means that every two vertices in a connected 3-balanced graph are contained in a cycle.

The next result strengthens Theorem \ref{thm: perfmatch} so as to allow a converse.


\begin{theorem} \label{thm: cubic sum characterization}
    Let $G$ be cubic.
    Then $G$ is 3-balanced if and only if $G$ has 3 disjoint perfect matchings labelled by elements of $\Z_3$ such that:
    \begin{itemize}
        \item For every cycle, $C = v_0 v_1  \ldots  v_{k-1}  v_0$, with corresponding edge labels $x_1,\dots, x_k$, the alternating sum 
                \[ S(C) = \sum_{i=1}^{k}(-1)^{i-1} x_i \]
            depends only on $v_0$.
    \end{itemize}
    In that case, $S(C) = 0$ if there is an even cycle passing through $v_0$.
\end{theorem}

\begin{proof}
    Suppose first that $G$ is 3-balanced. Use the induced edge coloring from Theorem \ref{thm: perfmatch} to color the edges so that there are 3 disjoint perfect matchings.
    Let $C = v_0 v_1  \cdots  v_{k-1}  v_0$ be a cycle in $G$ with corresponding edge labels $x_1, \dots, x_k$. When convenient, we will interpret the subscript on $v_i$ as an element of $\Z_k$.
    Then
    \begin{align*}
        \sum_{i=1}^{k}(-1)^{i-1} x_i
            &= \sum_{i=1}^{k}(-1)^{i-1} (\ell(v_{i-1})+\ell(v_i)) \\
            &= (1 + (-1)^{k-1}) \ell(v_0)
    \end{align*}
    and we are done.
    
    Conversely, suppose $\ell$ is a labeling of the edges of $G$ generating 3 disjoint perfect matchings that satisfy the property listed above. 
    For the same reasoning as found in the discussion before the proof, the Tait coloring here shows that every two distinct vertices of $G$ are contained in a cycle. 
    Note that we must have $S(C)=0$ if $k$ is even, by reversing the orientation of $C$. Similarly, if $k$ is odd but there exists an even cycle, $C_2$, containing $v_0$, then by hypothesis $S(C)=S(C_2)=0$.
    
    We will show that $\ell$ extends to a vertex labeling so that the edge labeling is the induced edge coloring. From this it will follow that $\ell$ is a 3-balanced vertex coloring.

    Fix $v_0 \in V$ and a cycle, $C = v_0 v_1  \cdots  v_{k-1}  v_0$, with corresponding edge labels $x_1, \dots, x_k$. 
    Define
    \begin{equation} \label{eqn: l(v_0) def}
        \ell(v_0) = (-1)^k \sum_{i=1}^{k}(-1)^{i-1} x_i.
    \end{equation}
    By hypothesis, $\ell(v_0)=0$ if there exists an even cycle through $v_0$.
    Inductively define $\ell(v_j)$, $1\leq j \leq k-1$, so that
    $\ell(v_{j-1}) + \ell(v_j) = x_j$.
    In particular,
    \begin{equation} \label{eqn: inductive def of l(v_j)}
        \ell(v_j) = (-1)^{j-1} \left( -\ell(v_0) + \sum_{i=1}^{j} (-1)^{i-1}                 x_i   \right)
    \end{equation}
    so that 
    \[  
        \ell(v_0) + \ell(v_{k-1}) =
        (1 + (-1)^{k+1}) \ell(v_0) 
        + (-1)^k \sum_{i=1}^{k-1}(-1)^{i-1} x_i. 
    \]
    When $k$ is even, $\ell(v_0) = 0$, so that $\sum_{i=1}^{k-1}(-1)^{i-1} x_i = x_k$. When $k$ is odd, 
    $(1 + (-1)^{k+1}) \ell(v_0) = \sum_{i=1}^{k}(-1)^{i-1} x_i$ so that, in either case, 
    \[ \ell(v_0) + \ell(v_{k-1}) = x_k. \]
    As a result, the $\ell$-induced edge labeling coincides with the original edge labeling on $C$.
    
    Suppose $C'=w_0w_1\cdots w_{m-1}w_0$ is also a cycle with $w_0=v_0$ and corresponding edge labelings $x_1',\ldots,x_m'$.
    Analogously use Equations \ref{eqn: l(v_0) def} and \ref{eqn: inductive def of l(v_j)} to define $\ell'(w_j)$, $0\leq j \leq m-1$, by replacing 
    $\ell$ by $\ell'$, $v_j$ by $w_j$, $k$ by $m$, and $x_i$ by $x_i'$ so
    that the $\ell'$-induced edge labeling coincides with the original edge labeling on $C'$.    
    By the hypothesis, $\ell'(v_0) = \ell(v_0)$ and both are 0 if there is an even cycle through $v_0$.

    If $C$ and $C'$ intersect at vertices besides $v_0$, choose $n$ minimal, $1\leq n \leq m-1$, so that $w_n = v_j$ for some $j$, $1\leq j \leq k-1$. 
    Consider the cycle of length $n+k-j$, $C''$, that travels along $C'$ from $w_0= v_0$ to $w_n = v_j$ and then along $C$ from $v_j= w_n$ up to $v_0=w_0$. By hypothesis, $S(C)=S(C'')$ so that 
    \begin{multline} \label{eqn: sum over two circles}
         \sum_{i=1}^{j}(-1)^{i-1} x_i + \sum_{i=j+1}^{k}(-1)^{i-1} x_i = \\
        \sum_{i=1}^{n}(-1)^{i-1} x_i' + 
                (-1)^{n-j} \sum_{i=j+1}^{k}(-1)^{i-1} x_{i}.
    \end{multline}
    Now if $n$ and $j$ have the same parity, then $(-1)^{n-j}=1$ and, by subtracting one side of Equation \ref{eqn: sum over two circles} from the other, we see that  
    \begin{equation} \label{eqn: parts of circles}
        (-1)^{n-j} \sum_{i=1}^{j}(-1)^{i-1} x_i = \sum_{i=1}^{n}(-1)^{i-1} x_i'.
    \end{equation}
    However, if $n$ and $j$ have opposite parities, then so do $k$ and $n+k-j$. That forces each side of Equation \ref{eqn: sum over two circles} to be zero since either $C$ or $C''$ will be an even cycle. 
    Noting $(-1)^{n-j}=-1$ in this case and solving for the first half of each side of Equation \ref{eqn: sum over two circles} being set to zero, we see that Equation \ref{eqn: parts of circles} holds in this case as well.

    With Equation \ref{eqn: parts of circles} in hand, we can calculate $ \ell'(w_n)$ with the analogue of Equation \ref{eqn: inductive def of l(v_j)} for $\ell'$ to get
    \[ \ell'(w_n) = (-1)^{n-1} \left( - \ell(v_0) + (-1)^{n-j} \sum_{i=1}^{j}(-1)^{i-1} x_i \right). \]
    If $n$ and $j$ have the same parity, Equation \ref{eqn: inductive def of l(v_j)} shows that $\ell'(v_j)=\ell(v_j)$.
    If they have opposite parities, then, as already seen, $\ell(v_0)=0$ and Equation \ref{eqn: inductive def of l(v_j)} again shows that $\ell'(v_j)=\ell(v_j)$.

    Similarly, a straightforward inductive argument on the number of intersections between $C$ and $C'$ shows that $\ell$ and $\ell'$ always agree on $C\cap C'$. As every two distinct vertices of $G$ are contained in a cycle, we can consistently extend $\ell$ to a coloring of $V$ with the desired induced edge coloring and are done.
\end{proof}

\begin{remark}
    It seems likely that Theorem \ref{thm: cubic sum characterization} can be generalized to a characterization of any 3-balanced graph by replacing the existence of 3 disjoint perfect matchings with a ``3-balanced edge coloring" and the alternating sum condition.
\end{remark}

The next criterion for being 3-balanced is more abstract, but may admit counting arguments. We begin with a definition.

\begin{definition} \label{def: c3bal data}
    We say that $(\{V_i\}, \{s_{ij}\})$ is a \emph{3-balanced cubic dataset} if
    \begin{enumerate}
        \item $\{V_i \mid i\in\Z_3\}$ are disjoint sets 
        \item $\{s_{ij}: V_i \rightarrow V_j \mid i,j\in\Z_3\}$ are bijections with $s_{ij}^{-1} = s_{ji}$
        \item $s_{ii}$, $i\in\Z_3$, has no fixed points.
    \end{enumerate}
\end{definition}

Note that Theorem \ref{thm: cubic bijection construction} below will show that $|V_i|=|V_j|\in 2\Z$, $i,j\in\Z_3$.

\begin{theorem} \label{thm: cubic bijection construction}
    Given a 3-balanced cubic dataset, $(\{V_i\}, \{s_{ij}\})$, form the graph $G=(V, E)$ with coloring $\ell$ by setting
    \begin{enumerate}
        \item $V = \coprod_{i=1}^3 V_i$
        \item $\ell(v) = i$ for $v\in V_i, i\in\Z_3$
        \item $E = \{ v s_{ij}(v) \mid v\in V_i, i,j\in\Z_3 \}$.
    \end{enumerate}
    Then $G$ is cubic and 3-balanced. Moreover, every cubic 3-balanced graph is of this form for some 3-balanced cubic dataset.
\end{theorem}

\begin{proof}
    Given a dataset, $(\{V_i\}, \{s_{ij}\})$, let $G$ and $\ell$ be constructed as in the statement of the theorem. It is straightforward to see that $G$ is cubic and 3-balanced.
    
    Conversely, to see subjectivity, 
    let $G$ be a 3-balanced graph with coloring $\ell$. 
    Define $V_i$, $i\in\Z_3$, as in Definition \ref{def: V_i and E_ij} so that $V = \coprod_{i=1}^3 V_i$,  $\ell(V_i)=\{i\}$, and $|V_i|=|V_j|$, $i,j\in\Z_3$, by Theorem \ref{thm: regular vertex number}. Finally, if $v\in V_i$, $i\in\Z_3$, and $j\in\Z_3$, let $s_{ij}(v)\in V_j$ be the unique vertex adjacent to $v_i$ colored by $j$. As $G$ is 3-balanced and cubic, $s_{ij}$ is a bijection and, by construction,$s_{ij}^{-1} = s_{ji}$ with $s_{ii}$ having no fixed points. As a result, $(\{V_i\}, \{s_{ij}\})$ is a dataset that generates $G$.
\end{proof}

\section{On the Number of Small and Large Cubic 3-Balanced Graphs} \label{sec: large and small cubic 3-bal}

In this section, we look at the number of cubic 3-balanced graphs for very small graphs and for very large ones.
We begin with classifying all connected cubic 3-balanced graphs on 6 vertices and 12 vertices. Note that Theorem \ref{thm: regular vertex number} shows that cubic 3-balanced graphs have orders that are divisible by 6.

We can easily see that both cubic graphs on 6 vertices, the triangular prism and $K_{3,3}$, are 3-balanced. See Figure \ref{fig: Cubic on 6}.

\begin{theorem} \label{thm: order 6 class}
    $K_{3,3}$ and the triangular prism are 3-balanced
    so that all connected cubic graphs on 6 vertices are 3-balanced.
\end{theorem}

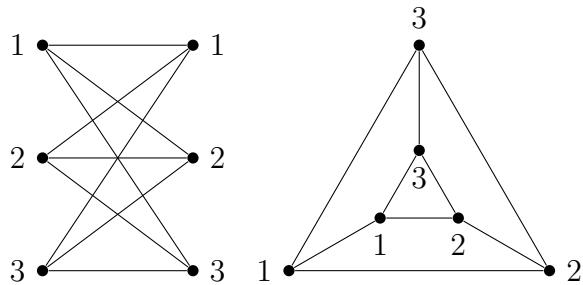
\begin{figure}[H]
    \centering

    \begin{tikzpicture} \label{Graph K33}
        \node[circle, fill, inner sep=1.5pt, label=left:1] (v1) at (0,1.5) {};
        \node[circle, fill, inner sep=1.5pt, label=left:2] (v2) at (0,0) {};
        \node[circle, fill, inner sep=1.5pt, label=left:3] (v3) at (0,-1.5) {};
        \node[circle, fill, inner sep=1.5pt, label=right:1] (u1) at (2,1.5) {};
        \node[circle, fill, inner sep=1.5pt, label=right:2] (u2) at (2,0) {};
        \node[circle, fill, inner sep=1.5pt, label=right:3] (u3) at (2,-1.5) {};
    
        \draw (v1) -- (u1);    \draw (v1) -- (u2);    \draw (v1) -- (u3);    \draw (v2) -- (u1);    \draw (v2) -- (u2);    \draw (v2) -- (u3);    \draw (v3) -- (u1);    \draw (v3) -- (u2);    \draw (v3) -- (u3);
    \end{tikzpicture}
    \begin{tikzpicture} \label{Triangular Prism}
        \def\n{3} 
        \def\m{3} 
        \def\radius{2cm} 
        \def\innerRadius{.6cm} 
              
        \node[circle, fill, inner sep=1.5pt, label=right:2] (v1) at ({360/\n * (0)-30}:\radius) {};
        \node[circle, fill, inner sep=1.5pt, label=above:3] (v2) at ({360/\n * (1)-30}:\radius) {};
        \node[circle, fill, inner sep=1.5pt, label=left:1] (v3) at ({360/\n * (2)-30}:\radius) {};
              
        \node[circle, fill, inner sep=1.5pt,label=below:2] (u1) at ({360/\m * (0)-30}:\innerRadius) {};
        \node[circle, fill, inner sep=1.5pt, label=below:3] (u2) at ({360/\m * (1)-30}:\innerRadius) {};
        \node[circle, fill, inner sep=1.5pt, label=below:1] (u3) at ({360/\m * (2)-30}:\innerRadius) {};
    
        \draw (v1) -- (u1);    \draw (v2) -- (u2);    \draw (v3) -- (u3);    \draw (v1) -- (v2);    \draw (v2) -- (v3);    \draw (v3) -- (v1);    \draw (u2) -- (u1);    \draw (u3) -- (u2);    \draw (u3) -- (u1);
    \end{tikzpicture}

    \caption{Cubic Graphs on 6 Vertices are 3-Balanced}
    \label{fig: Cubic on 6}
\end{figure}

Next we give a list of cubic 3-balanced forbidden subgraphs. See Figures \ref{fig:Diamond}, \ref{fig:Extended Bowtie}, \ref{fig:F1}, \ref{fig:F2}, \ref{fig:F3}, and \ref{fig:F4} for the notation $\{D, EB, F_1, \dots, F_4 \}$.

\begin{lemma} \label{lem: Forbidden Subgraphs}
    Cubic 3-balanced graphs cannot contain any subgraph from the set $\mathcal{F}=\{D, EB, F_1, \dots, F_4 \}$.
\end{lemma}

\begin{proof}
     These subgraphs are colored in Figure \ref{fig: Forbidden Subgraphs} as follows: Arbitrarily color a triangle or the neighborhood of a chosen vertex (indicated by a square), after which rest of the colors are forced. From these forced colorings it is straightforward to see that at least one neighborhood is not 3-balanced.
\end{proof}

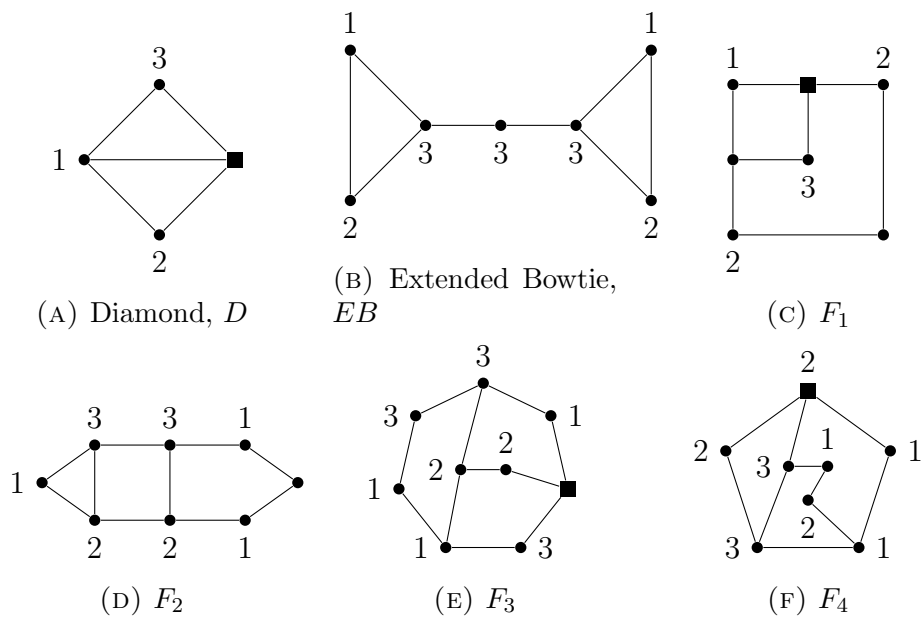
\begin{figure}[H]
\centering
\begin{subfigure}{0.3\textwidth}
    \begin{center}
        \begin{tikzpicture} 
        \node[circle, fill, inner sep=1.5pt,label=above:3] (v1) at (0,1) {};
        \node[circle, fill, inner sep=1.5pt,label=left:1] (v2) at (-1,0) {};
        \node[circle, fill, inner sep=1.5pt,label=below:2] (v3) at (0,-1) {};
        \node[rectangle, fill, inner sep=3pt] (v4) at (1,0) {};
       
        \draw (v1) -- (v2);    \draw (v2) -- (v3);    \draw (v3) -- (v4);    \draw (v4) -- (v1);    \draw (v2) -- (v4);
    \end{tikzpicture}
    \caption{Diamond, $D$}
    \label{fig:Diamond}
    \end{center}
\end{subfigure}
\hfill
\begin{subfigure}{0.3\textwidth}
    \begin{center}
        \begin{tikzpicture}
        \node[circle, fill, inner sep=1.5pt,label=below:3] (v1) at (-1,0) {};
        \node[circle, fill, inner sep=1.5pt,label=above:1] (v2) at (-2,1) {};
        \node[circle, fill, inner sep=1.5pt,label=below:2] (v3) at (-2,-1) {};
        \node[circle, fill, inner sep=1.5pt,label=below:3] (u1) at (1,0) {};
        \node[circle, fill, inner sep=1.5pt,label=above:1] (u2) at (2,1) {};
        \node[circle, fill, inner sep=1.5pt,label=below:2] (u3) at (2,-1) {};
        \node[circle, fill, inner sep=1.5pt,label=below:3] (w1) at (0,0) {};
    
        \draw (v1) -- (v2);    \draw (v1) -- (v3);    \draw (v3) -- (v2);    \draw (u1) -- (u2);    \draw (u1) -- (u3);    \draw (u3) -- (u2);    \draw (v1) -- (w1);    \draw (u1) -- (w1);
    \end{tikzpicture}
    \caption{Extended Bowtie, $EB$}
    \label{fig:Extended Bowtie}
    \end{center}
\end{subfigure}
\hfill
\begin{subfigure}{0.3\textwidth}
    \begin{center}
        \begin{tikzpicture} 
        \node[circle, fill, inner sep=1.5pt] (v1) at (0,0) {};
        \node[circle, fill, inner sep=1.5pt,label=above:1] (v2) at (0,1) {};
        \node[rectangle, fill, inner sep=3pt] (v3) at (1,1) {};
        \node[circle, fill, inner sep=1.5pt,label=below:3] (v4) at (1,0) {};
        \node[circle, fill, inner sep=1.5pt,label=above:2] (u1) at (2,1) {};
        \node[circle, fill, inner sep=1.5pt] (u2) at (2,-1) {};
        \node[circle, fill, inner sep=1.5pt,label=below:2] (u3) at (0,-1) {};
    
        \draw (v1) -- (v2);    \draw (v2) -- (v3);    \draw (v3) -- (v4);    \draw (v4) -- (v1);    \draw (v3) -- (u1);    \draw (u1) -- (u2);    \draw (u2) -- (u3);    \draw (v1) -- (u3);
    \end{tikzpicture}
    \caption{$F_1$}
    \label{fig:F1}
    \end{center}
\end{subfigure}
\hfill
\begin{subfigure}{0.3\textwidth}
    \begin{center}
        \begin{tikzpicture} 
        \node[circle, fill, inner sep=1.5pt,label=below:2] (v1) at (0,0) {};
        \node[circle, fill, inner sep=1.5pt,label=below:2] (v2) at (-1,0) {};
        \node[circle, fill, inner sep=1.5pt,label=left:1] (v3) at (-1.7,.5) {};
        \node[circle, fill, inner sep=1.5pt,label=above:3] (v4) at (-1,1) {};
        \node[circle, fill, inner sep=1.5pt,label=above:3] (v5) at (0,1) {};
        \node[circle, fill, inner sep=1.5pt,label=above:1] (v6) at (1,1) {};
        \node[circle, fill, inner sep=1.5pt] (v7) at (1.7,.5) {};
        \node[circle, fill, inner sep=1.5pt,label=below:1] (v8) at (1,0) {};
    
        \draw (v1) -- (v2);        \draw (v2) -- (v3);        \draw (v3) -- (v4);        \draw (v4) -- (v5);        \draw (v5) -- (v6);        \draw (v6) -- (v7);        \draw (v7) -- (v8);        \draw (v8) -- (v1);        \draw (v4) -- (v2);        \draw (v1) -- (v5);
    \end{tikzpicture}
    \caption{$F_2$}
    \label{fig:F2}
    \end{center}
\end{subfigure}
\hfill
\begin{subfigure}{0.3\textwidth}
    \begin{center}
        \begin{tikzpicture} 
    \def\n{7} 
    \def\m{2} 
    \def\radius{1.15cm} 
    \def\innerRadius{.3cm} 
          
    \node[rectangle, fill, inner sep=3pt] (v1) at ({360/\n * (0)-(720/7 -90)}:\radius) {};
    \node[circle, fill, inner sep=1.5pt,label=right:1] (v2) at ({360/\n * (1)-(720/7 -90)}:\radius) {};
    \node[circle, fill, inner sep=1.5pt,label=above:3] (v3) at ({360/\n * (2)-(720/7 -90)}:\radius) {};
    \node[circle, fill, inner sep=1.5pt,label=left:3] (v4) at ({360/\n * (3)-(720/7 -90)}:\radius) {};
    \node[circle, fill, inner sep=1.5pt,label=left:1] (v5) at ({360/\n * (4)-(720/7 -90))}:\radius) {};
    \node[circle, fill, inner sep=1.5pt,label=left:1] (v6) at ({360/\n * (5)-(720/7 -90)}:\radius) {};
    \node[circle, fill, inner sep=1.5pt,,label=right:3] (v7) at ({360/\n * (6)-(720/7 -90)}:\radius) {};
    
    \node[circle, fill, inner sep=1.5pt,label=above:2] (u1) at ({360/\m * (0)}:\innerRadius) {};
    \node[circle, fill, inner sep=1.5pt,label=left:2] (u2) at ({360/\m * (1)}:\innerRadius) {};

    \draw (v1) -- (v2);    \draw (v2) -- (v3);    \draw (v3) -- (v4);    \draw (v4) -- (v5);    \draw (v5) -- (v6);    \draw (v6) -- (v7);    \draw (v7) -- (v1);    \draw (v1) -- (u1);    \draw (u1) -- (u2);    \draw (v3) -- (u2);    \draw (v6) -- (u2);
\end{tikzpicture}
    \caption{$F_3$}
    \label{fig:F3}
    \end{center}
\end{subfigure}
\hfill
\begin{subfigure}{0.3\textwidth}
    \begin{center}
        \begin{tikzpicture} 
    \def\n{5} 
    \def\m{3} 
    \def\radius{1.15cm} 
    \def\innerRadius{.3cm} 
          
    \node[circle, fill, inner sep=1.5pt,label=right:1] (v1) at ({360/\n * (0)+18}:\radius) {};
    \node[rectangle, fill, inner sep=3pt,label=above:2] (v2) at ({360/\n * (1)+18}:\radius) {};
    \node[circle, fill, inner sep=1.5pt,label=left:2] (v3) at ({360/\n * (2)+18}:\radius) {};
    \node[circle, fill, inner sep=1.5pt,label=left:3] (v4) at ({360/\n * (3)+18}:\radius) {};
    \node[circle, fill, inner sep=1.5pt,label=right:1] (v5) at ({360/\n * (4))+18}:\radius) {};
    
    \node[circle, fill, inner sep=1.5pt,label=above:1] (u1) at ({360/\m * (0)+30}:\innerRadius) {};
    \node[circle, fill, inner sep=1.5pt,label=left:3] (u2) at ({360/\m * (1)+30}:\innerRadius) {};
    \node[circle, fill, inner sep=1.5pt,label=below:2] (u3) at ({360/\m * (2)+30}:\innerRadius) {};
          
    \draw (v1) -- (v2);    \draw (v2) -- (v3);    \draw (v3) -- (v4);    \draw (v4) -- (v5);    \draw (v5) -- (v1);
    \draw (u1) -- (u2);
    \draw (u1) -- (u3);
    \draw (u2) -- (v2);
    \draw (u2) -- (v4);
    \draw (u3) -- (v5);
\end{tikzpicture}
    \caption{$F_4$}
    \label{fig:F4}
    \end{center}
\end{subfigure}
        
\caption{Forbidden Subgraphs in 3-Balanced Cubic Graphs}
\label{fig: Forbidden Subgraphs}
\end{figure}

With Lemma \ref{lem: Forbidden Subgraphs} in hand, we can classify the connected cubic 3-balanced graphs on 12 vertices.

\begin{theorem} \label{thm: cubic on 12}
    There are 85 connected cubic graphs on 12 vertices, exactly 17 of which are 3-balanced. See the Appendix, \S\ref{sec: appendix}, for the list.
\end{theorem}

\begin{proof}
    As this proof mostly amounts to going through the list of the 85 connected cubic graphs, listed in \cite{Bussemaker}, we relegate the proof to the Appendix in \S\ref{sec: appendix}.
    We will explicitly note in the Appendix the graphs that are ruled out by each subgraph from Lemma \ref{lem: Forbidden Subgraphs}. We also give explicit colorings for the 3-balanced graphs. Also note that Tietze's Graph has chromatic index 4 and is thus not 3-balanced.
\end{proof}

We halt our investigation of small 3-balanced cubic graphs here, as there are 41,301 cubic graphs on 18 vertices, \cite{Bussemaker}.

We end on a note about the relative frequency of very large 3-balanced graphs by showing that they become relatively vanishingly rare. 
Let $\mathcal{G}_n$ denote the number of cubic 3-balanced graphs on $6n$ vertices and let $\mathcal{X}_n$ denote the number of cubic graphs on $6n$ vertices. 
Based off of work in \cite{DomaticNumber}, Aram \emph{et.~ al.~}~ investigate in \cite{TotalDomaticNumber} the \emph{total domatic number} of a graph, $d_t(G)$, the maximum number of total dominating sets into which the vertex set of $G$ can be partitioned. 
Note that a cubic graph is 3-balanced if and only if the total domatic number is 3. 
In \cite{TotalDomaticNumber}, Theorem 3.1, it is therefore shown that
\[  \lim_{n\rightarrow\infty} \frac{\mathcal{G}_n}{\mathcal{X}_n} = 0 \]
so that cubic 3-balanced graphs become relatively vanishingly rare.

\section{Concluding Remarks} \label{sec: concluding remarks}

The following is a list of questions for future work.

\begin{question}
    What is the number of non-isomorphic 3-balanced graphs on $n$ vertices?
\end{question}






\begin{question}
    Is it possible to classify all forbidden subgraphs of cubic 3-balanced graphs, besides those shown in Figure \ref{fig: Forbidden Subgraphs}, and is this list finite?
\end{question}

 
 \begin{question}
     What 6-regular graphs are 3-balanced, and what are the forbidden subgraphs for 6-regular graphs?
 \end{question}
 
See \cite{Meringer} for the number of non-isomorphic $r$-regular graphs on $n$ vertices. 

\begin{question}
    Are there characterizations of 6-regular graphs similar to Theorem \ref{thm: cubic sum characterization} and Theorem \ref{thm: cubic bijection construction}?
\end{question}

\begin{question}
    Are there any algebraic properties of 3-balanced graphs?
\end{question}

For example, some arc-transitive graphs are 3-balanced, e.g., the Pappus Graph, however, Tutte's 8-cage is not 3-balanced.

\begin{question}
    If $G_1 \nabla G_2$ is 3-balanced, are $G_1$ and $G_2$ are 3-balanced?
\end{question}

\begin{question}
    Do $3$-balanced graphs decompose into a sum of subgraphs each of which is 3-balanced with an equidistribution of colors (see Figure \ref{fig: Non-Regular 3-balanced graph})?
\end{question}

And finally, it would be interesting to investigate the notion of $k$-balanced for larger values of $k\in \Z^+$.

\newpage
\appendix
\section{Cubic Graphs on 12 Vertices} \label{sec: appendix}

In this appendix we present the details for Theorem \ref{thm: cubic on 12} by providing explicit colorings for the 17 3-balanced graphs, and by indicating why the other graphs are not 3-balanced. The graph numbers correspond with those from \cite{Bussemaker}. We start with the 17 3-balanced graphs:

\begin{center}
    \begin{longtable}{|m{.5cm} | m{1cm}| m{8cm} | } 
     \hline
     $G\#$ & Girth & Coloring \\ 
     \hline
     11&  3 & 
     \begin{center}
        \begin{tikzpicture} \label{Graph 11}
            \def\n{12} 
            \def\radius{2cm} 
            
            \draw (0,0) circle (\radius);
            \node[circle, fill, inner sep=1.5pt, label=right:2] (v1) at ({360/\n * (0)}:\radius) {};
            \node[circle, fill, inner sep=1.5pt, label=right:2] (v2) at ({360/\n * (1)}:\radius) {};
            \node[circle, fill, inner sep=1.5pt, label=above:1] (v3) at ({360/\n * (2)}:\radius) {};
            \node[circle, fill, inner sep=1.5pt, label=above:3] (v4) at ({360/\n * (3)}:\radius) {};
            \node[circle, fill, inner sep=1.5pt, label=above:2] (v5) at ({360/\n * (4)}:\radius) {};
            \node[circle, fill, inner sep=1.5pt, label=left:2] (v6) at ({360/\n * (5)}:\radius) {};
            \node[circle, fill, inner sep=1.5pt, label=left:3] (v7) at ({360/\n * (6)}:\radius) {};
            \node[circle, fill, inner sep=1.5pt, label=left:1] (v8) at ({360/\n * (7)}:\radius) {};
            \node[circle, fill, inner sep=1.5pt, label=below:1] (v9) at ({360/\n * (8)}:\radius) {};
            \node[circle, fill, inner sep=1.5pt, label=below:3] (v10) at ({360/\n * (9)}:\radius) {};
            \node[circle, fill, inner sep=1.5pt, label=below:3] (v11) at ({360/\n * (10)}:\radius) {};
            \node[circle, fill, inner sep=1.5pt, label=right:1] (v12) at ({360/\n * (11)}:\radius) {};
        
            \draw (v1) -- (v11);            \draw (v2) -- (v10);            \draw (v3) -- (v12);            \draw (v4) -- (v7);            \draw (v5) -- (v9);            \draw (v6) -- (v8);
            
        \end{tikzpicture} 
        \end{center}\\ 
     \hline
     13 & 3 & 
         \begin{center}
        \begin{tikzpicture} \label{Graph 13}
            \def\n{12} 
            \def\radius{2cm} 
            
            \draw (0,0) circle (\radius);
            \node[circle, fill, inner sep=1.5pt, label=right:2] (v1) at ({360/\n * (0)}:\radius) {};
            \node[circle, fill, inner sep=1.5pt, label=right:1] (v2) at ({360/\n * (1)}:\radius) {};
            \node[circle, fill, inner sep=1.5pt, label=above:1] (v3) at ({360/\n * (2)}:\radius) {};
            \node[circle, fill, inner sep=1.5pt, label=above:2] (v4) at ({360/\n * (3)}:\radius) {};
            \node[circle, fill, inner sep=1.5pt, label=above:1] (v5) at ({360/\n * (4)}:\radius) {};
            \node[circle, fill, inner sep=1.5pt, label=left:3] (v6) at ({360/\n * (5)}:\radius) {};
            \node[circle, fill, inner sep=1.5pt, label=left:3] (v7) at ({360/\n * (6)}:\radius) {};
            \node[circle, fill, inner sep=1.5pt, label=left:1] (v8) at ({360/\n * (7)}:\radius) {};
            \node[circle, fill, inner sep=1.5pt, label=below:1] (v9) at ({360/\n * (8)}:\radius) {};
            \node[circle, fill, inner sep=1.5pt, label=below:2] (v10) at ({360/\n * (9)}:\radius) {};
            \node[circle, fill, inner sep=1.5pt, label=below:3] (v11) at ({360/\n * (10)}:\radius) {};
            \node[circle, fill, inner sep=1.5pt, label=right:3] (v12) at ({360/\n * (11)}:\radius) {};
        
            \draw (v1) -- (v10);            \draw (v2) -- (v12);            \draw (v3) -- (v11);            \draw (v4) -- (v7);            \draw (v5) -- (v8);            \draw (v6) -- (v9);
            
        \end{tikzpicture} 
        \end{center}\\ 
     \hline
     14& 3& 
        \begin{center}
        \begin{tikzpicture} \label{Graph 14}
            \def\n{12} 
            \def\radius{2cm} 
            
            \draw (0,0) circle (\radius);
            \node[circle, fill, inner sep=1.5pt, label=right:2] (v1) at ({360/\n * (0)}:\radius) {};
            \node[circle, fill, inner sep=1.5pt, label=right:1] (v2) at ({360/\n * (1)}:\radius) {};
            \node[circle, fill, inner sep=1.5pt, label=above:3] (v3) at ({360/\n * (2)}:\radius) {};
            \node[circle, fill, inner sep=1.5pt, label=above:3] (v4) at ({360/\n * (3)}:\radius) {};
            \node[circle, fill, inner sep=1.5pt, label=above:1] (v5) at ({360/\n * (4)}:\radius) {};
            \node[circle, fill, inner sep=1.5pt, label=left:2] (v6) at ({360/\n * (5)}:\radius) {};
            \node[circle, fill, inner sep=1.5pt, label=left:2] (v7) at ({360/\n * (6)}:\radius) {};
            \node[circle, fill, inner sep=1.5pt, label=left:1] (v8) at ({360/\n * (7)}:\radius) {};
            \node[circle, fill, inner sep=1.5pt, label=below:3] (v9) at ({360/\n * (8)}:\radius) {};
            \node[circle, fill, inner sep=1.5pt, label=below:3] (v10) at ({360/\n * (9)}:\radius) {};
            \node[circle, fill, inner sep=1.5pt, label=below:1] (v11) at ({360/\n * (10)}:\radius) {};
            \node[circle, fill, inner sep=1.5pt, label=right:2] (v12) at ({360/\n * (11)}:\radius) {};
        
            \draw (v1) -- (v3);            \draw (v2) -- (v11);            \draw (v4) -- (v6);            \draw (v5) -- (v8);            \draw (v7) -- (v9);            \draw (v10) -- (v12);
            
        \end{tikzpicture} 
        \end{center}\\ 
     \hline
     15&  3 & 
     \begin{center}
        \begin{tikzpicture} \label{Graph 15}
            \def\n{12} 
            \def\radius{2cm} 
            
            \draw (0,0) circle (\radius);
            \node[circle, fill, inner sep=1.5pt, label=right:2] (v1) at ({360/\n * (0)}:\radius) {};
            \node[circle, fill, inner sep=1.5pt, label=right:2] (v2) at ({360/\n * (1)}:\radius) {};
            \node[circle, fill, inner sep=1.5pt, label=above:3] (v3) at ({360/\n * (2)}:\radius) {};
            \node[circle, fill, inner sep=1.5pt, label=above:3] (v4) at ({360/\n * (3)}:\radius) {};
            \node[circle, fill, inner sep=1.5pt, label=above:2] (v5) at ({360/\n * (4)}:\radius) {};
            \node[circle, fill, inner sep=1.5pt, label=left:1] (v6) at ({360/\n * (5)}:\radius) {};
            \node[circle, fill, inner sep=1.5pt, label=left:1] (v7) at ({360/\n * (6)}:\radius) {};
            \node[circle, fill, inner sep=1.5pt, label=left:2] (v8) at ({360/\n * (7)}:\radius) {};
            \node[circle, fill, inner sep=1.5pt, label=below:3] (v9) at ({360/\n * (8)}:\radius) {};
            \node[circle, fill, inner sep=1.5pt, label=below:3] (v10) at ({360/\n * (9)}:\radius) {};
            \node[circle, fill, inner sep=1.5pt, label=below:1] (v11) at ({360/\n * (10)}:\radius) {};
            \node[circle, fill, inner sep=1.5pt, label=right:1] (v12) at ({360/\n * (11)}:\radius) {};
        
            \draw (v1) -- (v10);            \draw (v2) -- (v11);            \draw (v3) -- (v12);            \draw (v4) -- (v6);            \draw (v5) -- (v8);            \draw (v7) -- (v9);
            
        \end{tikzpicture} 
        \end{center}\\ 
     \hline
     16&  4 & 
     \begin{center}
        \begin{tikzpicture} \label{Graph 16}
            \def\n{12} 
            \def\radius{2cm} 
            
            \draw (0,0) circle (\radius);
            \node[circle, fill, inner sep=1.5pt, label=right:1] (v1) at ({360/\n * (0)}:\radius) {};
            \node[circle, fill, inner sep=1.5pt, label=right:2] (v2) at ({360/\n * (1)}:\radius) {};
            \node[circle, fill, inner sep=1.5pt, label=above:3] (v3) at ({360/\n * (2)}:\radius) {};
            \node[circle, fill, inner sep=1.5pt, label=above:1] (v4) at ({360/\n * (3)}:\radius) {};
            \node[circle, fill, inner sep=1.5pt, label=above:2] (v5) at ({360/\n * (4)}:\radius) {};
            \node[circle, fill, inner sep=1.5pt, label=left:2] (v6) at ({360/\n * (5)}:\radius) {};
            \node[circle, fill, inner sep=1.5pt, label=left:1] (v7) at ({360/\n * (6)}:\radius) {};
            \node[circle, fill, inner sep=1.5pt, label=left:3] (v8) at ({360/\n * (7)}:\radius) {};
            \node[circle, fill, inner sep=1.5pt, label=below:3] (v9) at ({360/\n * (8)}:\radius) {};
            \node[circle, fill, inner sep=1.5pt, label=below:1] (v10) at ({360/\n * (9)}:\radius) {};
            \node[circle, fill, inner sep=1.5pt, label=below:2] (v11) at ({360/\n * (10)}:\radius) {};
            \node[circle, fill, inner sep=1.5pt, label=right:3] (v12) at ({360/\n * (11)}:\radius) {};
        
            \draw (v1) -- (v10);            \draw (v2) -- (v11);            \draw (v3) -- (v12);            \draw (v4) -- (v7);            \draw (v5) -- (v8);            \draw (v6) -- (v9);
            
        \end{tikzpicture} 
        \end{center}\\ 
     \hline
     33&  3& 
     \begin{center}
        \begin{tikzpicture} \label{Graph 33}
            \def\n{12} 
            \def\radius{2cm} 
            
            \draw (0,0) circle (\radius);
            \node[circle, fill, inner sep=1.5pt, label=right:1] (v1) at ({360/\n * (0)}:\radius) {};
            \node[circle, fill, inner sep=1.5pt, label=right:3] (v2) at ({360/\n * (1)}:\radius) {};
            \node[circle, fill, inner sep=1.5pt, label=above:3] (v3) at ({360/\n * (2)}:\radius) {};
            \node[circle, fill, inner sep=1.5pt, label=above:1] (v4) at ({360/\n * (3)}:\radius) {};
            \node[circle, fill, inner sep=1.5pt, label=above:2] (v5) at ({360/\n * (4)}:\radius) {};
            \node[circle, fill, inner sep=1.5pt, label=left:2] (v6) at ({360/\n * (5)}:\radius) {};
            \node[circle, fill, inner sep=1.5pt, label=left:1] (v7) at ({360/\n * (6)}:\radius) {};
            \node[circle, fill, inner sep=1.5pt, label=left:3] (v8) at ({360/\n * (7)}:\radius) {};
            \node[circle, fill, inner sep=1.5pt, label=below:3] (v9) at ({360/\n * (8)}:\radius) {};
            \node[circle, fill, inner sep=1.5pt, label=below:1] (v10) at ({360/\n * (9)}:\radius) {};
            \node[circle, fill, inner sep=1.5pt, label=below:2] (v11) at ({360/\n * (10)}:\radius) {};
            \node[circle, fill, inner sep=1.5pt, label=right:2] (v12) at ({360/\n * (11)}:\radius) {};
        
            \draw (v1) -- (v7);            \draw (v2) -- (v6);            \draw (v3) -- (v5);            \draw (v4) -- (v10);            \draw (v8) -- (v12);            \draw (v9) -- (v11);
            
        \end{tikzpicture} 
        \end{center}\\ 
     \hline
     34&  3& 
     \begin{center}
        \begin{tikzpicture} \label{Graph 34}
            \def\n{12} 
            \def\radius{2cm} 
            
            \draw (0,0) circle (\radius);
            \node[circle, fill, inner sep=1.5pt, label=right:1] (v1) at ({360/\n * (0)}:\radius) {};
            \node[circle, fill, inner sep=1.5pt, label=right:3] (v2) at ({360/\n * (1)}:\radius) {};
            \node[circle, fill, inner sep=1.5pt, label=above:3] (v3) at ({360/\n * (2)}:\radius) {};
            \node[circle, fill, inner sep=1.5pt, label=above:1] (v4) at ({360/\n * (3)}:\radius) {};
            \node[circle, fill, inner sep=1.5pt, label=above:2] (v5) at ({360/\n * (4)}:\radius) {};
            \node[circle, fill, inner sep=1.5pt, label=left:2] (v6) at ({360/\n * (5)}:\radius) {};
            \node[circle, fill, inner sep=1.5pt, label=left:1] (v7) at ({360/\n * (6)}:\radius) {};
            \node[circle, fill, inner sep=1.5pt, label=left:3] (v8) at ({360/\n * (7)}:\radius) {};
            \node[circle, fill, inner sep=1.5pt, label=below:3] (v9) at ({360/\n * (8)}:\radius) {};
            \node[circle, fill, inner sep=1.5pt, label=below:1] (v10) at ({360/\n * (9)}:\radius) {};
            \node[circle, fill, inner sep=1.5pt, label=below:2] (v11) at ({360/\n * (10)}:\radius) {};
            \node[circle, fill, inner sep=1.5pt, label=right:2] (v12) at ({360/\n * (11)}:\radius) {};
        
            \draw (v1) -- (v7);            \draw (v2) -- (v6);            \draw (v3) -- (v5);            \draw (v4) -- (v10);            \draw (v8) -- (v11);            \draw (v9) -- (v12);
            
        \end{tikzpicture} 
        \end{center}\\ 
     \hline
     35&  3& 
     \begin{center}
        \begin{tikzpicture} \label{Graph 35}
            \def\n{12} 
            \def\radius{2cm} 
            
            \draw (0,0) circle (\radius);
            \node[circle, fill, inner sep=1.5pt, label=right:1] (v1) at ({360/\n * (0)}:\radius) {};
            \node[circle, fill, inner sep=1.5pt, label=right:3] (v2) at ({360/\n * (1)}:\radius) {};
            \node[circle, fill, inner sep=1.5pt, label=above:3] (v3) at ({360/\n * (2)}:\radius) {};
            \node[circle, fill, inner sep=1.5pt, label=above:1] (v4) at ({360/\n * (3)}:\radius) {};
            \node[circle, fill, inner sep=1.5pt, label=above:2] (v5) at ({360/\n * (4)}:\radius) {};
            \node[circle, fill, inner sep=1.5pt, label=left:2] (v6) at ({360/\n * (5)}:\radius) {};
            \node[circle, fill, inner sep=1.5pt, label=left:1] (v7) at ({360/\n * (6)}:\radius) {};
            \node[circle, fill, inner sep=1.5pt, label=left:3] (v8) at ({360/\n * (7)}:\radius) {};
            \node[circle, fill, inner sep=1.5pt, label=below:3] (v9) at ({360/\n * (8)}:\radius) {};
            \node[circle, fill, inner sep=1.5pt, label=below:1] (v10) at ({360/\n * (9)}:\radius) {};
            \node[circle, fill, inner sep=1.5pt, label=below:2] (v11) at ({360/\n * (10)}:\radius) {};
            \node[circle, fill, inner sep=1.5pt, label=right:2] (v12) at ({360/\n * (11)}:\radius) {};
        
            \draw (v1) -- (v7);            \draw (v2) -- (v5);            \draw (v3) -- (v6);            \draw (v4) -- (v10);            \draw (v8) -- (v11);            \draw (v9) -- (v12);
            
        \end{tikzpicture} 
        \end{center}\\ 
     \hline
     46&  3& 
     \begin{center}
        \begin{tikzpicture} \label{Graph 46}
            \def\n{12} 
            \def\radius{2cm} 
            
            \draw (0,0) circle (\radius);
            \node[circle, fill, inner sep=1.5pt, label=right:2] (v1) at ({360/\n * (0)}:\radius) {};
            \node[circle, fill, inner sep=1.5pt, label=right:1] (v2) at ({360/\n * (1)}:\radius) {};
            \node[circle, fill, inner sep=1.5pt, label=above:1] (v3) at ({360/\n * (2)}:\radius) {};
            \node[circle, fill, inner sep=1.5pt, label=above:3] (v4) at ({360/\n * (3)}:\radius) {};
            \node[circle, fill, inner sep=1.5pt, label=above:2] (v5) at ({360/\n * (4)}:\radius) {};
            \node[circle, fill, inner sep=1.5pt, label=left:1] (v6) at ({360/\n * (5)}:\radius) {};
            \node[circle, fill, inner sep=1.5pt, label=left:1] (v7) at ({360/\n * (6)}:\radius) {};
            \node[circle, fill, inner sep=1.5pt, label=left:3] (v8) at ({360/\n * (7)}:\radius) {};
            \node[circle, fill, inner sep=1.5pt, label=below:2] (v9) at ({360/\n * (8)}:\radius) {};
            \node[circle, fill, inner sep=1.5pt, label=below:2] (v10) at ({360/\n * (9)}:\radius) {};
            \node[circle, fill, inner sep=1.5pt, label=below:3] (v11) at ({360/\n * (10)}:\radius) {};
            \node[circle, fill, inner sep=1.5pt, label=right:3] (v12) at ({360/\n * (11)}:\radius) {};
        
            \draw (v1) -- (v5);            \draw (v2) -- (v12);            \draw (v3) -- (v10);            \draw (v4) -- (v8);            \draw (v6) -- (v11);            \draw (v7) -- (v9);
            
        \end{tikzpicture} 
        \end{center}\\ 
     \hline
     53&  3& 
     \begin{center}
        \begin{tikzpicture} \label{Graph 53}
            \def\n{12} 
            \def\radius{2cm} 
            
            \draw (0,0) circle (\radius);
            \node[circle, fill, inner sep=1.5pt, label=right:1] (v1) at ({360/\n * (0)}:\radius) {};
            \node[circle, fill, inner sep=1.5pt, label=right:3] (v2) at ({360/\n * (1)}:\radius) {};
            \node[circle, fill, inner sep=1.5pt, label=above:3] (v3) at ({360/\n * (2)}:\radius) {};
            \node[circle, fill, inner sep=1.5pt, label=above:1] (v4) at ({360/\n * (3)}:\radius) {};
            \node[circle, fill, inner sep=1.5pt, label=above:2] (v5) at ({360/\n * (4)}:\radius) {};
            \node[circle, fill, inner sep=1.5pt, label=left:2] (v6) at ({360/\n * (5)}:\radius) {};
            \node[circle, fill, inner sep=1.5pt, label=left:1] (v7) at ({360/\n * (6)}:\radius) {};
            \node[circle, fill, inner sep=1.5pt, label=left:3] (v8) at ({360/\n * (7)}:\radius) {};
            \node[circle, fill, inner sep=1.5pt, label=below:3] (v9) at ({360/\n * (8)}:\radius) {};
            \node[circle, fill, inner sep=1.5pt, label=below:1] (v10) at ({360/\n * (9)}:\radius) {};
            \node[circle, fill, inner sep=1.5pt, label=below:2] (v11) at ({360/\n * (10)}:\radius) {};
            \node[circle, fill, inner sep=1.5pt, label=right:2] (v12) at ({360/\n * (11)}:\radius) {};
        
            \draw (v1) -- (v7);            \draw (v2) -- (v11);            \draw (v3) -- (v5);            \draw (v4) -- (v10);            \draw (v6) -- (v9);            \draw (v8) -- (v12);
            
        \end{tikzpicture} 
        \end{center}\\ 
     \hline
     60&  3& 
     \begin{center}
        \begin{tikzpicture} \label{Graph 60}
            \def\n{12} 
            \def\radius{2cm} 
            
            \draw (0,0) circle (\radius);
            \node[circle, fill, inner sep=1.5pt, label=right:1] (v1) at ({360/\n * (0)}:\radius) {};
            \node[circle, fill, inner sep=1.5pt, label=right:3] (v2) at ({360/\n * (1)}:\radius) {};
            \node[circle, fill, inner sep=1.5pt, label=above:3] (v3) at ({360/\n * (2)}:\radius) {};
            \node[circle, fill, inner sep=1.5pt, label=above:1] (v4) at ({360/\n * (3)}:\radius) {};
            \node[circle, fill, inner sep=1.5pt, label=above:2] (v5) at ({360/\n * (4)}:\radius) {};
            \node[circle, fill, inner sep=1.5pt, label=left:2] (v6) at ({360/\n * (5)}:\radius) {};
            \node[circle, fill, inner sep=1.5pt, label=left:1] (v7) at ({360/\n * (6)}:\radius) {};
            \node[circle, fill, inner sep=1.5pt, label=left:3] (v8) at ({360/\n * (7)}:\radius) {};
            \node[circle, fill, inner sep=1.5pt, label=below:3] (v9) at ({360/\n * (8)}:\radius) {};
            \node[circle, fill, inner sep=1.5pt, label=below:1] (v10) at ({360/\n * (9)}:\radius) {};
            \node[circle, fill, inner sep=1.5pt, label=below:2] (v11) at ({360/\n * (10)}:\radius) {};
            \node[circle, fill, inner sep=1.5pt, label=right:2] (v12) at ({360/\n * (11)}:\radius) {};
        
            \draw (v1) -- (v7);            \draw (v2) -- (v12);            \draw (v3) -- (v5);            \draw (v4) -- (v10);            \draw (v6) -- (v9);            \draw (v8) -- (v11);
            
        \end{tikzpicture} 
        \end{center}\\ 
     \hline
     63&  3& 
     \begin{center}
        \begin{tikzpicture} \label{Graph 63}
            \def\n{12} 
            \def\radius{2cm} 
            
            \draw (0,0) circle (\radius);
            \node[circle, fill, inner sep=1.5pt, label=right:1] (v1) at ({360/\n * (0)}:\radius) {};
            \node[circle, fill, inner sep=1.5pt, label=right:3] (v2) at ({360/\n * (1)}:\radius) {};
            \node[circle, fill, inner sep=1.5pt, label=above:3] (v3) at ({360/\n * (2)}:\radius) {};
            \node[circle, fill, inner sep=1.5pt, label=above:1] (v4) at ({360/\n * (3)}:\radius) {};
            \node[circle, fill, inner sep=1.5pt, label=above:2] (v5) at ({360/\n * (4)}:\radius) {};
            \node[circle, fill, inner sep=1.5pt, label=left:2] (v6) at ({360/\n * (5)}:\radius) {};
            \node[circle, fill, inner sep=1.5pt, label=left:1] (v7) at ({360/\n * (6)}:\radius) {};
            \node[circle, fill, inner sep=1.5pt, label=left:3] (v8) at ({360/\n * (7)}:\radius) {};
            \node[circle, fill, inner sep=1.5pt, label=below:3] (v9) at ({360/\n * (8)}:\radius) {};
            \node[circle, fill, inner sep=1.5pt, label=below:1] (v10) at ({360/\n * (9)}:\radius) {};
            \node[circle, fill, inner sep=1.5pt, label=below:2] (v11) at ({360/\n * (10)}:\radius) {};
            \node[circle, fill, inner sep=1.5pt, label=right:2] (v12) at ({360/\n * (11)}:\radius) {};
        
            \draw (v1) -- (v7);            \draw (v2) -- (v12);            \draw (v3) -- (v5);            \draw (v4) -- (v10);            \draw (v6) -- (v8);            \draw (v9) -- (v11);
            
        \end{tikzpicture} 
        \end{center}\\ 
     \hline
     66&  4& 
     \begin{center}
        \begin{tikzpicture} \label{Graph 66}
            \def\n{12} 
            \def\radius{2cm} 
            
            \draw (0,0) circle (\radius);
            \node[circle, fill, inner sep=1.5pt, label=right:3] (v1) at ({360/\n * (0)}:\radius) {};
            \node[circle, fill, inner sep=1.5pt, label=right:3] (v2) at ({360/\n * (1)}:\radius) {};
            \node[circle, fill, inner sep=1.5pt, label=above:1] (v3) at ({360/\n * (2)}:\radius) {};
            \node[circle, fill, inner sep=1.5pt, label=above:1] (v4) at ({360/\n * (3)}:\radius) {};
            \node[circle, fill, inner sep=1.5pt, label=above:2] (v5) at ({360/\n * (4)}:\radius) {};
            \node[circle, fill, inner sep=1.5pt, label=left:2] (v6) at ({360/\n * (5)}:\radius) {};
            \node[circle, fill, inner sep=1.5pt, label=left:3] (v7) at ({360/\n * (6)}:\radius) {};
            \node[circle, fill, inner sep=1.5pt, label=left:3] (v8) at ({360/\n * (7)}:\radius) {};
            \node[circle, fill, inner sep=1.5pt, label=below:1] (v9) at ({360/\n * (8)}:\radius) {};
            \node[circle, fill, inner sep=1.5pt, label=below:1] (v10) at ({360/\n * (9)}:\radius) {};
            \node[circle, fill, inner sep=1.5pt, label=below:2] (v11) at ({360/\n * (10)}:\radius) {};
            \node[circle, fill, inner sep=1.5pt, label=right:2] (v12) at ({360/\n * (11)}:\radius) {};
        
            \draw (v1) -- (v10);            \draw (v2) -- (v5);            \draw (v3) -- (v12);            \draw (v4) -- (v7);            \draw (v6) -- (v9);            \draw (v8) -- (v11);
            
        \end{tikzpicture} 
        \end{center}\\ 
     \hline
     67&  4& 
     \begin{center}
        \begin{tikzpicture} \label{Graph 67}
            \def\n{12} 
            \def\radius{2cm} 
            
            \draw (0,0) circle (\radius);
            \node[circle, fill, inner sep=1.5pt, label=right:3] (v1) at ({360/\n * (0)}:\radius) {};
            \node[circle, fill, inner sep=1.5pt, label=right:1] (v2) at ({360/\n * (1)}:\radius) {};
            \node[circle, fill, inner sep=1.5pt, label=above:2] (v3) at ({360/\n * (2)}:\radius) {};
            \node[circle, fill, inner sep=1.5pt, label=above:3] (v4) at ({360/\n * (3)}:\radius) {};
            \node[circle, fill, inner sep=1.5pt, label=above:1] (v5) at ({360/\n * (4)}:\radius) {};
            \node[circle, fill, inner sep=1.5pt, label=left:2] (v6) at ({360/\n * (5)}:\radius) {};
            \node[circle, fill, inner sep=1.5pt, label=left:3] (v7) at ({360/\n * (6)}:\radius) {};
            \node[circle, fill, inner sep=1.5pt, label=left:1] (v8) at ({360/\n * (7)}:\radius) {};
            \node[circle, fill, inner sep=1.5pt, label=below:2] (v9) at ({360/\n * (8)}:\radius) {};
            \node[circle, fill, inner sep=1.5pt, label=below:3] (v10) at ({360/\n * (9)}:\radius) {};
            \node[circle, fill, inner sep=1.5pt, label=below:1] (v11) at ({360/\n * (10)}:\radius) {};
            \node[circle, fill, inner sep=1.5pt, label=right:2] (v12) at ({360/\n * (11)}:\radius) {};
        
            \draw (v1) -- (v7);            \draw (v2) -- (v8);            \draw (v3) -- (v9);            \draw (v4) -- (v10);            \draw (v5) -- (v11);            \draw (v6) -- (v12);
            
        \end{tikzpicture} 
        \end{center}\\ 
     \hline
     68&  4& 
     \begin{center}
        \begin{tikzpicture} \label{Graph 68}
            \def\n{12} 
            \def\radius{2cm} 
            
            \draw (0,0) circle (\radius);
            \node[circle, fill, inner sep=1.5pt, label=right:1] (v1) at ({360/\n * (0)}:\radius) {};
            \node[circle, fill, inner sep=1.5pt, label=right:3] (v2) at ({360/\n * (1)}:\radius) {};
            \node[circle, fill, inner sep=1.5pt, label=above:2] (v3) at ({360/\n * (2)}:\radius) {};
            \node[circle, fill, inner sep=1.5pt, label=above:1] (v4) at ({360/\n * (3)}:\radius) {};
            \node[circle, fill, inner sep=1.5pt, label=above:1] (v5) at ({360/\n * (4)}:\radius) {};
            \node[circle, fill, inner sep=1.5pt, label=left:2] (v6) at ({360/\n * (5)}:\radius) {};
            \node[circle, fill, inner sep=1.5pt, label=left:3] (v7) at ({360/\n * (6)}:\radius) {};
            \node[circle, fill, inner sep=1.5pt, label=left:1] (v8) at ({360/\n * (7)}:\radius) {};
            \node[circle, fill, inner sep=1.5pt, label=below:2] (v9) at ({360/\n * (8)}:\radius) {};
            \node[circle, fill, inner sep=1.5pt, label=below:3] (v10) at ({360/\n * (9)}:\radius) {};
            \node[circle, fill, inner sep=1.5pt, label=below:3] (v11) at ({360/\n * (10)}:\radius) {};
            \node[circle, fill, inner sep=1.5pt, label=right:2] (v12) at ({360/\n * (11)}:\radius) {};
        
            \draw (v1) -- (v8);            \draw (v2) -- (v7);            \draw (v3) -- (v9);            \draw (v4) -- (v10);            \draw (v5) -- (v11);            \draw (v6) -- (v12);
            
        \end{tikzpicture} 
        \end{center}\\ 
     \hline
     70&  4& 
     \begin{center}
        \begin{tikzpicture} \label{Graph 70}
            \def\n{12} 
            \def\radius{2cm} 
            
            \draw (0,0) circle (\radius);
            \node[circle, fill, inner sep=1.5pt, label=right:3] (v1) at ({360/\n * (0)}:\radius) {};
            \node[circle, fill, inner sep=1.5pt, label=right:3] (v2) at ({360/\n * (1)}:\radius) {};
            \node[circle, fill, inner sep=1.5pt, label=above:1] (v3) at ({360/\n * (2)}:\radius) {};
            \node[circle, fill, inner sep=1.5pt, label=above:2] (v4) at ({360/\n * (3)}:\radius) {};
            \node[circle, fill, inner sep=1.5pt, label=above:2] (v5) at ({360/\n * (4)}:\radius) {};
            \node[circle, fill, inner sep=1.5pt, label=left:1] (v6) at ({360/\n * (5)}:\radius) {};
            \node[circle, fill, inner sep=1.5pt, label=left:1] (v7) at ({360/\n * (6)}:\radius) {};
            \node[circle, fill, inner sep=1.5pt, label=left:2] (v8) at ({360/\n * (7)}:\radius) {};
            \node[circle, fill, inner sep=1.5pt, label=below:3] (v9) at ({360/\n * (8)}:\radius) {};
            \node[circle, fill, inner sep=1.5pt, label=below:3] (v10) at ({360/\n * (9)}:\radius) {};
            \node[circle, fill, inner sep=1.5pt, label=below:1] (v11) at ({360/\n * (10)}:\radius) {};
            \node[circle, fill, inner sep=1.5pt, label=right:2] (v12) at ({360/\n * (11)}:\radius) {};
        
            \draw (v1) -- (v7);            \draw (v2) -- (v5);            \draw (v3) -- (v11);            \draw (v4) -- (v10);            \draw (v6) -- (v9);            \draw (v8) -- (v12);
            
        \end{tikzpicture} 
        \end{center}\\ 
     \hline
     72&  4& 
     \begin{center}
        \begin{tikzpicture} \label{Graph 72}
            \def\n{12} 
            \def\radius{2cm} 
            
            \draw (0,0) circle (\radius);
            \node[circle, fill, inner sep=1.5pt, label=right:3] (v1) at ({360/\n * (0)}:\radius) {};
            \node[circle, fill, inner sep=1.5pt, label=right:1] (v2) at ({360/\n * (1)}:\radius) {};
            \node[circle, fill, inner sep=1.5pt, label=above:2] (v3) at ({360/\n * (2)}:\radius) {};
            \node[circle, fill, inner sep=1.5pt, label=above:3] (v4) at ({360/\n * (3)}:\radius) {};
            \node[circle, fill, inner sep=1.5pt, label=above:1] (v5) at ({360/\n * (4)}:\radius) {};
            \node[circle, fill, inner sep=1.5pt, label=left:1] (v6) at ({360/\n * (5)}:\radius) {};
            \node[circle, fill, inner sep=1.5pt, label=left:3] (v7) at ({360/\n * (6)}:\radius) {};
            \node[circle, fill, inner sep=1.5pt, label=left:2] (v8) at ({360/\n * (7)}:\radius) {};
            \node[circle, fill, inner sep=1.5pt, label=below:1] (v9) at ({360/\n * (8)}:\radius) {};
            \node[circle, fill, inner sep=1.5pt, label=below:3] (v10) at ({360/\n * (9)}:\radius) {};
            \node[circle, fill, inner sep=1.5pt, label=below:2] (v11) at ({360/\n * (10)}:\radius) {};
            \node[circle, fill, inner sep=1.5pt, label=right:2] (v12) at ({360/\n * (11)}:\radius) {};
        
            \draw (v1) -- (v7);            \draw (v2) -- (v9);            \draw (v3) -- (v8);            \draw (v4) -- (v10);            \draw (v5) -- (v12);            \draw (v6) -- (v11);
            
        \end{tikzpicture} 
        \end{center}\\ 
     \hline
     \end{longtable}
\end{center}

Now we move on to all of the other cubic graphs on 12 vertices. We list forbidden subgraphs as pictured in Theorem \ref{thm: cubic on 12} and all the graphs that contain them. Note that Graph 74 is Tietze's Graph.

\begin{center}
\begin{longtable}{|m{2cm} | m{7cm} | } 
     \hline
     Forbidden Subgraph & Graph Numbers, \cite{Bussemaker} \\ 
     \hline
     Bridge& 1-4 \\
     \hline
     Diamond&  5-10 and 17-28\\
     \hline
     Extended Bowtie & 12, 30, 31, 36, 37, 39, 42, 48, 50-52, 56, and 65 \\
     \hline
     $F_1$ & 29, 32, 38, 41, 44, 45, 47, 49, 55, 57, 59, 62, 64, 69, 71, 73, 75-81, and 83\\
     \hline
     $F_2$ & 43\\
     \hline
     $F_3$ & 40, 61, 82, 84, and 85\\
     \hline
     $F_4$ & 54 and 58\\
     \hline
     
 \end{longtable}
\end{center}

\bibliographystyle{abbrvnat}
\bibliography{refs}

\begin{thebibliography}{13}
\providecommand{\natexlab}[1]{#1}
\providecommand{\url}[1]{\texttt{#1}}
\expandafter\ifx\csname urlstyle\endcsname\relax
  \providecommand{\doi}[1]{doi: #1}\else
  \providecommand{\doi}{doi: \begingroup \urlstyle{rm}\Url}\fi

\bibitem[Aram et~al.(2012)Aram, Sheikholeslami, and Volkmann]{TotalDomaticNumber}
H.~Aram, S.~M. Sheikholeslami, and L.~Volkmann.
\newblock On the total domatic number of regular graphs.
\newblock \emph{Trans. Comb.}, 1\penalty0 (1):\penalty0 45--51, 2012.
\newblock ISSN 2251-8657,2251-8665.

\bibitem[Biswas and Das(2022)]{GeneralizedPappus}
S.~Biswas and A.~Das.
\newblock A generalization of {P}appus graph.
\newblock \emph{Electron. J. Graph Theory Appl. (EJGTA)}, 10\penalty0 (1):\penalty0 345--356, 2022.
\newblock ISSN 2338-2287.
\newblock \doi{10.5614/ejgta.2022.10.1.25}.
\newblock URL \url{https://doi.org/10.5614/ejgta.2022.10.1.25}.

\bibitem[Bussemaker et~al.(1977)Bussemaker, \v~Cobelji\'c, Cvetkovi\'c, and Seidel]{Bussemaker}
F.~C. Bussemaker, S.~\v~Cobelji\'c, D.~M. Cvetkovi\'c, and J.~J. Seidel.
\newblock Cubic graphs on {$\leq 14$} vertices.
\newblock \emph{J. Combinatorial Theory Ser. B}, 23\penalty0 (2-3):\penalty0 234--235, 1977.
\newblock ISSN 0095-8956.
\newblock \doi{10.1016/0095-8956(77)90034-x}.
\newblock URL \url{https://doi.org/10.1016/0095-8956(77)90034-x}.

\bibitem[Castagna and Prins(1972)]{CastagnaPrins}
F.~Castagna and G.~Prins.
\newblock Every generalized {P}etersen graph has a {T}ait coloring.
\newblock \emph{Pacific J. Math.}, 40:\penalty0 53--58, 1972.
\newblock ISSN 0030-8730,1945-5844.
\newblock URL \url{http://projecteuclid.org/euclid.pjm/1102968820}.

\bibitem[Chen et~al.(2015)Chen, Kim, Tait, and Verstraete]{Coupon}
B.~Chen, J.~H. Kim, M.~Tait, and J.~Verstraete.
\newblock On coupon colorings of graphs.
\newblock \emph{Discrete Appl. Math.}, 193:\penalty0 94--101, 2015.
\newblock ISSN 0166-218X,1872-6771.
\newblock \doi{10.1016/j.dam.2015.04.026}.
\newblock URL \url{https://doi.org/10.1016/j.dam.2015.04.026}.

\bibitem[Dankelmann and Calkin(2004)]{DomaticNumber}
P.~Dankelmann and N.~Calkin.
\newblock The domatic number of regular graphs.
\newblock \emph{Ars Combin.}, 73:\penalty0 247--255, 2004.
\newblock ISSN 0381-7032,2817-5204.

\bibitem[Freyberg and Marr(2024)]{FreybergMarr}
B.~Freyberg and A.~Marr.
\newblock Neighborhood balanced colorings of graphs.
\newblock \emph{Graphs Combin.}, 40\penalty0 (2):\penalty0 Paper No. 41, 13, 2024.
\newblock ISSN 0911-0119,1435-5914.
\newblock \doi{10.1007/s00373-024-02766-9}.
\newblock URL \url{https://doi.org/10.1007/s00373-024-02766-9}.

\bibitem[Hahn et~al.(2002)Hahn, Kratochv\'il, \v~Sir\'a\v~n, and Sotteau]{InjectiveChromaticNumber}
G.~n. Hahn, J.~Kratochv\'il, J.~\v~Sir\'a\v~n, and D.~Sotteau.
\newblock On the injective chromatic number of graphs.
\newblock \emph{Discrete Math.}, 256\penalty0 (1-2):\penalty0 179--192, 2002.
\newblock ISSN 0012-365X,1872-681X.
\newblock \doi{10.1016/S0012-365X(01)00466-6}.
\newblock URL \url{https://doi.org/10.1016/S0012-365X(01)00466-6}.

\bibitem[Isaacs(1975)]{Isaacs}
R.~Isaacs.
\newblock Infinite families of nontrivial trivalent graphs which are not {T}ait colorable.
\newblock \emph{Amer. Math. Monthly}, 82:\penalty0 221--239, 1975.
\newblock ISSN 0002-9890,1930-0972.
\newblock \doi{10.2307/2319844}.
\newblock URL \url{https://doi.org/10.2307/2319844}.

\bibitem[Li et~al.(2020)Li, Shao, and Zhu]{InjectiveGP}
Z.~Li, Z.~Shao, and E.~Zhu.
\newblock Injective coloring of generalized {P}etersen graphs.
\newblock \emph{Houston J. Math.}, 46\penalty0 (1):\penalty0 1--12, 2020.
\newblock ISSN 0362-1588.

\bibitem[Mann(1965)]{Mann}
H.~B. Mann.
\newblock On linear relations between roots of unity.
\newblock \emph{Mathematika}, 12:\penalty0 107--117, 1965.
\newblock ISSN 0025-5793.
\newblock \doi{10.1112/S0025579300005210}.
\newblock URL \url{https://doi.org/10.1112/S0025579300005210}.

\bibitem[Meringer(1999)]{Meringer}
M.~Meringer.
\newblock Fast generation of regular graphs and construction of cages.
\newblock \emph{J. Graph Theory}, 30\penalty0 (2):\penalty0 137--146, 1999.
\newblock ISSN 0364-9024,1097-0118.
\newblock \doi{10.1002/(SICI)1097-0118(199902)30:2<137::AID-JGT7>3.0.CO;2-G}.

\bibitem[Watkins(1969)]{Watkins}
M.~E. Watkins.
\newblock A theorem on {T}ait colorings with an application to the generalized {P}etersen graphs.
\newblock \emph{J. Combinatorial Theory}, 6:\penalty0 152--164, 1969.
\newblock ISSN 0021-9800.

\end{thebibliography}
\end{document}